\documentclass[twoside,11pt]{article}
\usepackage[nohyperref,preprint]{jmlr2e}
\usepackage[hidelinks]{hyperref}
\usepackage{mathtools}
\usepackage{amssymb}
\usepackage[caption=false,font=footnotesize,labelformat=simple]{subfig}

\newcommand{\RR}{\mathbb R}
\newcommand{\ee}{\mathrm e}
\DeclareMathOperator*{\st}{s.t.}
\DeclareMathOperator*{\rank}{rank}
\DeclareMathOperator*{\diag}{diag}
\DeclareMathOperator*{\tr}{tr}
\DeclareMathOperator*{\vect}{vec}
\DeclareMathOperator*{\mat}{mat_S}
\DeclareMathOperator{\RIP}{RIP}
\DeclareMathOperator{\BDP}{BDP}
\DeclareMathOperator{\OPT}{OPT}
\DeclareMathOperator{\LMI}{LMI}
\DeclarePairedDelimiter{\abs}{\lvert}{\rvert}
\DeclarePairedDelimiter{\norm}{\lVert}{\rVert}

\title{Global and Local Analyses of Nonlinear Low-Rank Matrix Recovery Problems}

\author{\name Yingjie Bi \email yingjiebi@berkeley.edu \\
\addr Department of Industrial Engineering and Operations Research\\
University of California\\
Berkeley, CA 94720, USA
\AND
\name Javad Lavaei \email lavaei@berkeley.edu \\
\addr Department of Industrial Engineering and Operations Research\\
University of California\\
Berkeley, CA 94720, USA
\thanks{A preliminary version of this paper has appeared in \cite{BL2020}. Compared with the conference paper, we have developed a major new result on the local guarantee for the absence of spurious local minima in the general rank-$r$ case and included a new application from machine learning to illustrate the effectiveness of our results.}}

\ShortHeadings{Global and Local Analyses of Nonlinear Low-Rank Matrix Recovery Problems}{Bi and Lavaei}
\firstpageno{1}

\begin{document}
\maketitle

\begin{abstract}
The restricted isometry property (RIP) is a well-known condition that guarantees the absence of spurious local minima in low-rank matrix recovery problems with linear measurements. In this paper, we introduce a novel property named bound difference property (BDP) to study low-rank matrix recovery problems with nonlinear measurements. Using RIP and BDP jointly, we first focus on the rank-1 matrix recovery problem, for which we propose a new criterion to certify the nonexistence of spurious local minima over the entire space. We then analyze the general case with an arbitrary rank and derive a condition to rule out the possibility of having a spurious solution in a ball around the true solution. The developed conditions lead to much stronger theoretical guarantees than the existing bounds on RIP.
\end{abstract}

\section{Introduction}\label{sec:intro}

The \emph{low-rank matrix recovery problem} plays a central role in many machine learning problems, such as recommendation systems \citep{KBV2009} and motion detection \citep{ZYY2013,FS2020}. It also appears in engineering problems, such as power system state estimation \citep{ZML2018}. The goal of this problem is to recover an unknown low-rank matrix $M^* \in \RR^{n \times n}$ from certain measurements of the entries of $M^*$. The measurement equations may be linear or nonlinear, which will be discussed separately in the next two subsections.

\subsection{Low-Rank Matrix Recovery with Linear Measurements}

The basic form of the low-rank matrix recovery problem is the symmetric and noiseless one with linear measurements and the quadratic loss. The linear measurements can be represented by a linear operator $\mathcal A: \RR^{n \times n} \to \RR^m$ given by
\[
\mathcal A(M)=(\langle A_1,M \rangle,\dots,\langle A_m,M \rangle)^T.
\]
The ground-truth matrix $M^*$ is assumed to be symmetric and positive semidefinite with $\rank(M^*) \leq r$. The recovery problem can be formulated as follows:
\begin{equation}\label{eq:recorig}
\begin{aligned}
\min \quad & \frac{1}{2}\norm{\mathcal A(M)-d}^2 \\
\st \quad & \rank(M) \leq r, \\
\quad & M \succeq 0, \quad M \in \RR^{n \times n},
\end{aligned}
\end{equation}
where $d=\mathcal A(M^*)$. By factoring the decision variable $M$ into its low-rank factors $XX^T$, the above problem can be rewritten as the unconstrained problem:
\begin{equation}\label{eq:linearrec}
\min_{X \in \RR^{n \times r}}\left\{\frac{1}{2}\norm{\mathcal A(XX^T)-d}^2\right\}.
\end{equation}

The optimization \eqref{eq:linearrec} is commonly solved by local search methods. Since \eqref{eq:linearrec} is generally nonconvex, local search methods may converge to a spurious local minimum (a non-global local minimum is called a spurious solution). To provide theoretical guarantees on the performance of local search methods for the low-rank matrix recovery, several papers have developed various conditions under which the optimization \eqref{eq:linearrec} is free of spurious local minima. In what follows, we will briefly review the state-of-the-art results on this problem.

Given a linear operator $\mathcal A$, define its corresponding quadratic form $\mathcal Q: \RR^{n \times n} \times \RR^{n \times n} \to \RR$ as
\begin{equation}\label{eq:linearquad}
[\mathcal Q](K,L)=\langle \mathcal A(K),\mathcal A(L) \rangle,
\end{equation}
for all $K,L \in \RR^{n \times n}$.

\begin{definition}[\cite{RFP2010}]\label{def:linearrip}
A quadratic form $\mathcal Q:\RR^{n \times n} \times \RR^{n \times n} \to \RR$ satisfies the \emph{restricted isometry property} (RIP) of rank $2r$ for a constant $\delta \in [0,1)$, denoted as $\delta$-$\RIP_{2r}$, if
\[
(1-\delta)\norm{K}_F^2 \leq [\mathcal Q](K,K) \leq (1+\delta)\norm{K}_F^2
\]
for all matrices $K \in \RR^{n \times n}$ with $\rank(K) \leq 2r$.
\end{definition}

\cite{GJZ2017} showed that the problem \eqref{eq:linearrec} has no spurious local minima if the quadratic form $\mathcal Q$ satisfies $\delta$-$\RIP_{2r}$ with $\delta<1/5$. \cite{ZSL2019} strengthened this result for the special case of $r=1$ by showing that $\delta$-$\RIP_{2r}$ with $\delta<1/2$ is sufficient to guarantee the absence of spurious local minima for \eqref{eq:linearrec}. \cite{ZJSL2018} provided an example with a spurious local minimum in case of $\delta=1/2$ to support the tightness of the bound.

When $\delta \geq 1/5$ in the case $r>1$ or $\delta \geq 1/2$ in the case $r=1$, the $\delta$-$\RIP_{2r}$ property can still be useful in the sense that they can lead to \emph{local} guarantees for the absence of spurious local minima in the problem \eqref{eq:linearrec}, rather than \emph{global} guarantees. This means that there is no spurious local minimizer $X$ as long as $XX^T$ is in a neighborhood of $M^*$. Under a local guarantee, local search methods would still converge to the global optimal solution if they are initialized sufficiently close to the ground truth. Many techniques for finding such a good initial point have been developed in the literature (see Section~\ref{sec:related} for a brief discussion). In the rank-1 case, \cite{ZSL2019} proved the following local guarantee for the absence of spurious local minima:

\begin{theorem}[\cite{ZSL2019}]\label{thm:prvlocalrank1}
Assume that $r=1$ and the quadratic form $\mathcal Q$ satisfies the $\delta$-$\RIP_2$ property for some constant $\delta$ such that
\[
\delta<\sqrt{1-\frac{\epsilon^2}{2(1-\epsilon)}}
\]
with $0<\epsilon \leq (\sqrt 5-1)/2$. Then, the problem \eqref{eq:linearrec} has no spurious local minimizer $X$ that satisfies
\[
\norm{XX^T-M^*}_F \leq \epsilon\norm{M^*}_F.
\]
\end{theorem}

Note that $\epsilon\norm{M^*}_F$ defines the radius of the ball around the ground truth that is devoid of spurious solutions. The recent work \citep{ZZ2020} generalized the techniques in \cite{ZSL2019}, which led to the following result that can be applied to the rank-$r$ case for any $r$ but is weaker than Theorem~\ref{thm:prvlocalrank1} in the rank-1 case (rewritten here in an equivalent form):

\begin{theorem}[\cite{ZZ2020}]\label{thm:prvlocalrankr}
Assume that the quadratic form $\mathcal Q$ satisfies the $\delta$-$\RIP_{2r}$ property for some constant $\delta$ such that $
\delta<\sqrt{1-\epsilon}$ with $0<\epsilon \leq 1$. Then, the problem \eqref{eq:linearrec} has no spurious local minimizer $X$ that satisfies
\[
\norm{XX^T-M^*}_F \leq \epsilon\lambda_r(M^*).
\]
\end{theorem}

\subsection{Nonlinear Low-Rank Matrix Recovery}

Given the above-mentioned results for the low-rank matrix recovery problems with linear measurements, it is natural to investigate whether these results can be extended to problems that are similar to \eqref{eq:linearrec} but have more complex objective functions. The purpose of this paper is to study the existence of spurious local minima for the general low-rank matrix recovery problem
\begin{equation}\label{eq:nonlinearrec}
\min_{X \in \RR^{n \times r}}f(XX^T),
\end{equation}
where $f: \RR^{n \times n} \to \RR$ is an arbitrary function. Problems of this form are abound in many machine learning tasks (see Section~\ref{sec:app} for an application). Moreover, every polynomial optimization problem can be formulated as such, and therefore the analysis of \eqref{eq:nonlinearrec} enables the design of global optimization techniques for nonconvex polynomial optimization \citep{MSFL2017}. In this paper, $f$ is always assumed to be twice continuously differentiable. The problem \eqref{eq:linearrec} is a special case of \eqref{eq:nonlinearrec} by choosing
\begin{equation}\label{eq:linearobj}
f(M)=\frac{1}{2}\norm{\mathcal A(M)-d}^2.
\end{equation}
In the case with linear measurements, note that $f(M^*)=0$ and therefore $M^*$ is a global minimizer of $f$. In other words, there are often infinitely many minimizers for $f$, but the goal is to find the ground-truth low-rank solution $M^*$. Similar to the linear measurement case, we assume that the problem \eqref{eq:nonlinearrec} has a ground truth $M^*=ZZ^T$ with $\rank(M^*) \leq r$ that is a global minimizer of $f(M)$.

The Hessian of the function $f$ in \eqref{eq:nonlinearrec}, denoted as $\nabla^2f(M)$, can be also regarded as a quadratic form whose action on any two matrices $K,L \in \RR^{n \times n}$ is given by
\[
[\nabla^2f(M)](K,L)=\sum_{i,j,k,l=1}^n\frac{\partial^2f}{\partial M_{ij}\partial M_{kl}}(M)K_{ij}L_{kl}.
\]
If $f$ is considered to be the special function in \eqref{eq:linearobj}, then its corresponding Hessian $\nabla^2f(M)$ becomes exactly the quadratic form $\mathcal Q$ defined in \eqref{eq:linearquad}. Therefore, we naturally extend the definition of the $\delta$-$\RIP_{2r}$ property for quadratic forms given in Definition~\ref{def:linearrip} to general functions $f$ by restricting their Hessian.

\begin{definition}\label{def:nonlinearrip}
A twice continuously differentiable function $f:\RR^{n \times n} \to \RR$ satisfies the \emph{restricted isometry property} of rank $2r$ for a constant $\delta \in [0,1)$, denoted as $\delta$-$\RIP_{2r}$, if
\begin{equation}\label{eq:nonlinearrip}
(1-\delta)\norm{K}_F^2 \leq [\nabla^2f(M)](K,K) \leq (1+\delta)\norm{K}_F^2
\end{equation}
for all matrices $M,K \in \RR^{n \times n}$ with $\rank(M) \leq 2r$ and $\rank(K) \leq 2r$.
\end{definition}

It is still unknown whether the $\delta$-$\RIP_{2r}$ condition could lead to the nonexistence of spurious local minima. However, \cite{LZT2019} proved that the problem \eqref{eq:nonlinearrec} has no spurious local minima under a stronger condition, named $\delta$-$\RIP_{2r,4r}$ with $\delta<1/5$, as defined below.

\begin{definition}\label{def:nonlinearrip4r}
A twice continuously differentiable function $f:\RR^{n \times n} \to \RR$ satisfies the \emph{restricted isometry property} of rank $(2r,4r)$ for a constant $\delta \in [0,1)$, denoted as $\delta$-$\RIP_{2r,4r}$, if
\[
(1-\delta)\norm{K}_F^2 \leq [\nabla^2f(M)](K,K) \leq (1+\delta)\norm{K}_F^2
\]
for all matrices $M,K \in \RR^{n \times n}$ with $\rank(M) \leq 2r$ and $\rank(K) \leq 4r$.
\end{definition}

For the general recovery problem \eqref{eq:nonlinearrec} with $r=1$, the previous results in \cite{ZSL2019} and \cite{LZT2019} both have serious limitations. The bound $\delta<1/2$ given in \cite{ZSL2019} is proven to be tight in the case when $f$ is generated by linear measurements, but it is not applicable to nonlinear measurements. The bound $\delta<1/5$ given in \cite{LZT2019} can be applied to a general function $f$, but it is not tight even in the linear case. To address these issues, we develop a new criterion to guarantee the absence of spurious local minima globally in \eqref{eq:nonlinearrec} for a general function $f$ in the rank-$1$ case, which is more powerful than the previous conditions. Unlike the bound given in \cite{LZT2019}, our new criterion completely depends on the properties of the Hessian of the function $f$ applied to rank-$2$ matrices, rather than rank-$4$ matrices. Note that the rank-$1$ case has applications in many problems, such as motion detection \citep{FS2020} and power system state estimation \citep{ZML2018}.

For the problem \eqref{eq:nonlinearrec} with $r>1$, we also present a new local guarantee for the absence of spurious local minima near the ground truth $M^*$. Our work not only offers the first result in the literature on the local guarantee for general nonlinear problems, but also exhibits an improvement over the previous results stated in Theorem~\ref{thm:prvlocalrank1} and Theorem~\ref{thm:prvlocalrankr} for problems with linear measurements.

\subsection{Notations}

$I_n$ is the identity matrix of size $n \times n$, and $\diag(a_1,\dots,a_n)$ is a diagonal matrix whose diagonal entries are $a_1,\dots,a_n$. $\mathbf A=\vect A$ is the vector obtained from stacking the columns of a matrix $A$. Given a vector $\mathbf A \in \RR^{n^2}$, define its symmetric matricization $\mat\mathbf A=(A+A^T)/2$, where $A \in \RR^{n \times n}$ is the unique matrix satisfying $\mathbf A=\vect A$. $A \otimes B$ is the Kronecker product of $A$ and $B$, which satisfies the well-known identity:
\[
\vect(AXB^T)=(B \otimes A)\vect X.
\]
For two matrices $A,B$ of the same size, define their inner product
\[
\langle A,B \rangle=\tr(A^TB)=\langle \vect A,\vect B\rangle,
\]
and let $\norm{A}_F=\sqrt{\langle A,A\rangle}$ denote the Frobenius norm of the matrix $A$, where $\tr(\cdot)$ is the trace operator. Moreover, $\norm{v}$ is the Euclidean norm of the vector $v$. For a square matrix $A \in \RR^{n \times n}$, $A \succeq 0$ means that $A$ is symmetric and positive semidefinite. Let
\[
\lambda_1(A) \geq \lambda_2(A) \geq \dots \geq \lambda_n(A)
\]
denote the eigenvalues of $A$ sorted in a decreasing order.

\section{Related Works}\label{sec:related}

The classical approach for solving low-rank matrix recovery problems is through convex relaxations. A semidefinite program can be obtained by removing the nonconvex low-rank constraint in \eqref{eq:recorig}. Since the seminal work by \cite{RFP2010}, there is a plethora of researches on deriving conditions under which the convex relaxation is able to recover the exact solution of the original nonconvex problem. Most of the proposed conditions are based on RIP, including \cite{CP2011,CSV2013,CZ2013,ZL2017,LZMV2020}. Another direction is to show that the convex relaxation is exact with high probability if the measurements are random and have sufficient number of samples \citep{CR2009,CT2010}. The major drawback of the convex relaxation approach is that semidefinite programs are expensive to solve for large-scale problems arising in machine learning.

An alternative approach for solving low-rank matrix recovery problems is based on local search methods, such as gradient descent algorithms \citep{RS2005,LRST2010,RR2013,GHJY2015,TBSS2016}, iterative hard thresholding \citep{RSS2017} and trust-region methods \citep{SQW2016,BAC2019}. They can be efficiently applied to large-scale problems, but the quality of the obtained solution depends on whether the objective has a benign landscape. As previously discussed in Section~\ref{sec:intro}, RIP-type conditions can be used to guarantee the absence of spurious local minima over the entire space. Under these conditions, any local search method that converges to a local minimum will be able to recover the globally optimal solution. The existing proof techniques for the analysis of spurious local minima can be roughly categorized into two groups: 1) checking whether an arbitrary matrix $X$ is a spurious local minimizer for a given measurement operator $\mathcal A$ and a given ground truth matrix $M^*$ \citep{BNS2016,GLM2016,GJZ2017,PKCS2017,ZWYG2018,LZT2019}; 2) checking whether a measurement operator $\mathcal A$ exists that makes a given point $X$ a spurious local solution for a given ground truth matrix $M^*$. The second proof technique was first proposed in \cite{ZSL2019} and later extended in \cite{MSL2020,ZZ2020} for structured operators. We adopt the same technique in the current paper.

In the case when the absence of spurious local minima in the entire space cannot be guaranteed, there are various approaches to handle the problem: 1) apply special initialization schemes such as spectral methods to find an initial point sufficiently close to the ground truth \citep{ZL2015,CLS2015,BKS2016,SL2016,PKCS2018}; 2) use randomized algorithms such as stochastic descent methods to escape saddle points or poor local minimizers \citep{GHJY2015}; or 3) initialize the algorithm randomly multiple times \citep{GS2018,ZZ2020}.

\section{Main Results}

To obtain a powerful condition for guaranteeing the absence of spurious local minima in problem \eqref{eq:nonlinearrec}, it is helpful to shed light on a distinguishing property of the function in \eqref{eq:linearobj} for linear measurements that does not hold in the general case: the Hessian matrices at all points are equal. If a general function $f$ satisfies $\delta$-$\RIP_{2r}$, \eqref{eq:nonlinearrip} intuitively states that the Hessian $\nabla^2f(M)$ should be close to the quadratic form defined by an identity matrix, at least when applied to rank-$2r$ matrices. Hence, $\nabla^2f(M)$ should change slowly when $M$ alters. The above discussion motivates the introduction of a new notion below.

\begin{definition}
A twice continuously differentiable function $f:\RR^{n \times n} \to \RR$ satisfies the \emph{bounded difference property} of rank $2r$ for a constant $\kappa \geq 0$, denoted as $\kappa$-$\BDP_{2r}$, if
\begin{equation}\label{eq:bdpineq}
\abs{[\nabla^2f(M)-\nabla^2f(M')](K,L)} \leq \kappa\norm{K}_F\norm{L}_F
\end{equation}
for all matrices $M,M',K,L \in \RR^{n \times n}$ whose ranks are at most $2r$.
\end{definition}

It turns out that the RIP and BDP properties are not fully independent. Their relationship is summarized in the following theorems that will be proved in Section~\ref{sec:props}.

\begin{theorem}\label{thm:rip2rbdp}
If the function $f$ satisfies $\delta$-$\RIP_{2r}$, then it also satisfies $4\delta$-$\BDP_{2r}$.
\end{theorem}

\begin{theorem}\label{thm:rip4rbdp}
If the function $f$ satisfies $\delta$-$\RIP_{2r,4r}$, then it also satisfies $2\delta$-$\BDP_{2r}$.
\end{theorem}

The bounds in the above two theorems are tight. In Section~\ref{sec:props}, we will construct a class of functions $f$ that satisfy the $\delta$-$\RIP_{2r}$ property but do not satisfy the $\kappa$-$\BDP_{2r}$ property for some $\kappa$ with $\kappa/\delta$ being arbitrarily close to 4. Similar examples can also be constructed for Theorem~\ref{thm:rip4rbdp}.

The main results of this paper will be stated below, which are powerful criteria for the global and local nonexistence of spurious local minima. The proofs are given in Section~\ref{sec:nospurious}.

\begin{theorem}[Global Guarantee for $r=1$]\label{thm:main}
When $r=1$, the problem \eqref{eq:nonlinearrec} has no spurious local minima if the function $f$ satisfies the $\delta$-$\RIP_2$ and $\kappa$-$\BDP_2$ properties for some constants $\delta$ and $\kappa$ such that
\[
\delta<\frac{2-6(1+\sqrt 2)\kappa}{4+6(1+\sqrt 2)\kappa}.
\]
\end{theorem}

\begin{theorem}[Local Guarantee for $r \geq 1$]\label{thm:local}
Assume that the function $f$ satisfies the $\delta$-$\RIP_{2r}$ property for some constant $\delta$ such that
\[
\delta<\sqrt{1-\frac{3+2\sqrt 2}{4}\epsilon^2}
\]
with $0<\epsilon \leq 2(\sqrt 2-1)$. Then, the problem \eqref{eq:nonlinearrec} has no spurious local minimizer $X$ that satisfies
\[
\norm{XX^T-M^*}_F \leq \epsilon\lambda_r(M^*).
\]
\end{theorem}

In the case of linear measurements and the quadratic loss, the function $f$ satisfies the $\kappa$-$\BDP_{2r}$ property with $\kappa=0$. Hence, Theorem~\ref{thm:main} immediately recovers the result in \cite{ZSL2019} stating that the problem \eqref{eq:linearrec} with $r=1$ has no spurious local minima if the quadratic form $\mathcal Q$ satisfies the $\delta$-$\RIP_2$ property with $\delta<1/2$.

As a by-product, Theorem~\ref{thm:local} also improves the existing local guarantees summarized in Theorem~\ref{thm:prvlocalrank1} and Theorem~\ref{thm:prvlocalrankr} for certain linear cases. If $r=1$, Theorem~\ref{thm:local} can possibly offer a region free of spurious local minima that is larger than the region obtained from Theorem~\ref{thm:prvlocalrank1}. The reason is that $\epsilon$ in Theorem~\ref{thm:prvlocalrank1} is capped at $(\sqrt 5-1)/2$, which is increased to $2(\sqrt 2-1)$ in Theorem~\ref{thm:local} (note that $\lambda_r(M^*)=\norm{M^*}_F$ if $r=1$). For an arbitrary rank, Theorem~\ref{thm:local} strengthens the result of Theorem~\ref{thm:prvlocalrankr} in terms of the order of the bound as a function of $\epsilon$, when $\epsilon$ is small.

Theorem~\ref{thm:main} and Theorem~\ref{thm:local} are even more powerful for functions $f$ associated with nonlinear measurements. At the end of Section~\ref{sec:props} and in Section~\ref{sec:app}, we will offer such examples for which the absence of spurious local minima can be certified by Theorem~\ref{thm:main} or Theorem~\ref{thm:local}, while the existing conditions in the literature fail to work.

\section{RIP and BDP Properties}\label{sec:props}

In this section, the relationship among the $\RIP_{2r}$, $\RIP_{2r,4r}$ and $\BDP_{2r}$ properties of a given function $f$ will be investigated. We will first prove Theorem~\ref{thm:rip2rbdp} and Theorem~\ref{thm:rip4rbdp}, and then show that the bounds in these theorems are tight. The following lemma will be needed, which appears in \cite{Candes2008,BNS2016,LZT2019} under different notations. We include a short proof here for completeness.

\begin{lemma}\label{lem:rop}
If a quadratic form $\mathcal Q$ satisfies $\delta$-$\RIP_{2r}$,
then
\[
\abs{[\mathcal Q](K,L)-\langle K,L\rangle} \leq \delta\norm{K}_F\norm{L}_F
\]
for all matrices $K,L \in \RR^{n \times n}$ with $\rank(K) \leq r$, $\rank(L) \leq r$.
\end{lemma}
\begin{proof}
Without loss of generality, assume that $\norm{K}_F=\norm{L}_F=1$. By the $\delta$-$\RIP_{2r}$ property of $\mathcal Q$, we have
\begin{align*}
(1-\delta)\norm{K-L}_F^2 &\leq [\mathcal Q](K-L,K-L) \leq (1+\delta)\norm{K-L}_F^2, \\
(1-\delta)\norm{K+L}_F^2 &\leq [\mathcal Q](K+L,K+L) \leq (1+\delta)\norm{K+L}_F^2.
\end{align*}
Taking the difference between the above two inequalities, one can obtain
\begin{align*}
4[\mathcal Q](K,L) &\leq (1+\delta)\norm{K+L}_F^2-(1-\delta)\norm{K-L}_F^2=4\delta+4\langle K,L \rangle, \\
-4[\mathcal Q](K,L) &\leq (1+\delta)\norm{K-L}_F^2-(1-\delta)\norm{K+L}_F^2=4\delta-4\langle K,L \rangle,
\end{align*}
which proves the desired inequality.
\end{proof}

\par\noindent{\bf Proof of Theorem~\ref{thm:rip4rbdp}\ }
Let $M$ and $M'$ be two matrices of rank at most $2r$. By the definition of $\delta$-$\RIP_{2r,4r}$ of the function $f$, both $\nabla^2f(M)$ and $\nabla^2f(M')$ satisfy $\delta$-$\RIP_{4r}$. After the constant $r$ in the statement of Lemma~\ref{lem:rop} is replaced by $2r$, we obtain
\begin{gather*}
\abs{[\nabla^2f(M)](K,L)-\langle K,L\rangle} \leq \delta\norm{K}_F\norm{L}_F, \\
\abs{[\nabla^2f(M')](K,L)-\langle K,L\rangle} \leq \delta\norm{K}_F\norm{L}_F,
\end{gather*}
for all matrices $K,L \in \RR^{n \times n}$ of rank at most $2r$, which leads to \eqref{eq:bdpineq} for $\kappa=2\delta$.
\hfill\BlackBox\\[2mm]
\par\noindent{\bf Proof of Theorem~\ref{thm:rip2rbdp}\ }
We first prove that any quadratic form $\mathcal Q$ with $\delta$-$\RIP_{2r}$ satisfies
\begin{equation}\label{eq:rop2r}
\abs{[\mathcal Q](K,L)-\langle K,L\rangle} \leq 2\delta\norm{K}_F\norm{L}_F,
\end{equation}
for all matrices $K,L \in \RR^{n \times n}$ of rank at most $2r$. Let $K=UDV^T$ be the singular value decomposition of $K$. Write $D=D_1+D_2$ in which $D_1$ and $D_2$ both have at most $r$ nonzero entries, and let $K_1=UD_1V^T$ and $K_2=UD_2V^T$. Then, $K=K_1+K_2$, where $\rank(K_1) \leq r$, $\rank(K_2) \leq r$ and $\langle K_1,K_2 \rangle=0$. We decompose $L=L_1+L_2$ similarly. By Lemma~\ref{lem:rop}, it holds that
\begin{align*}
\abs{[\mathcal Q](K,L)-\langle K,L \rangle} &\leq \abs{[\mathcal Q](K_1,L_1)-\langle K_1,L_1 \rangle}+\abs{[\mathcal Q](K_1,L_2)-\langle K_1,L_2 \rangle} \\
&+\abs{[\mathcal Q](K_2,L_1)-\langle K_2,L_1 \rangle}+\abs{[\mathcal Q](K_2,L_2)-\langle K_2,L_2 \rangle} \\
&\leq \delta(\norm{K_1}_F+\norm{K_2}_F)(\norm{L_1}_F+\norm{L_2}_F) \\
&\leq 2\delta\sqrt{\norm{K_1}_F^2+\norm{K_2}_F^2}\sqrt{\norm{L_1}_F^2+\norm{L_2}_F^2} \\
&=2\delta\norm{K}_F\norm{L}_F.
\end{align*}
The remaining proof is exactly the same as the proof of Theorem~\ref{thm:rip4rbdp}.
\hfill\BlackBox\\[2mm]

The inequality \eqref{eq:rop2r} is parallel to the \emph{square root lifting inequality} \citep{CWX2009} in the compressed sensing problem. Our result can be regarded as a generalization of that result to the low-rank matrix recovery problem.

In what follows, we will show that the bounds in Theorem~\ref{thm:rip2rbdp} and Theorem~\ref{thm:rip4rbdp} are tight. To this end, we will work on examples of function $f$ with $\delta$-$\RIP_{2r}$ or $\delta$-$\RIP_{4r}$ for a small $\delta$ whose Hessian has a large variation across different points. Consider an integer $n \geq 4$ and an integer $r \geq 1$. Let
\[
A_1=\frac{1}{\sqrt n}\diag(a_1,\dots,a_n)
\]
with $a_i \in \{-1,1\}$ whose exact value will be determined later. One can extend $A_1$ to an orthonormal basis $A_1,\dots,A_{n^2}$ of the space $\RR^{n \times n}$. Define a linear operator $\mathcal A: \RR^{n \times n} \to \RR^{n^2-1}$ by letting
\[
\mathcal A(M)=(\langle A_2,M \rangle, \dots, \langle A_{n^2},M \rangle).
\]
Then, for every matrix $M \in \RR^{n \times n}$, it holds that
\[
\norm{\mathcal A(M)}^2=\norm{M}_F^2-(\langle A_1,M \rangle)^2 \leq \norm{M}_F^2.
\]
Now, assume that $M$ is a matrix with $\rank(M) \leq 2r$, and let $\sigma_1(M),\dots,\sigma_{2r}(M)$ denote its $2r$ largest singular values. Observe that
\begin{equation*}
\abs{\langle A_1,M \rangle} \leq \frac{1}{\sqrt n}\sum_{i=1}^n\abs{M_{ii}} \leq \frac{1}{\sqrt n}\sum_{i=1}^{2r}\sigma_i(M)=\sqrt\frac{2r}{n}\sqrt{\sum_{i=1}^{2r}\sigma_i^2(M)}=\sqrt\frac{2r}{n}\norm{M}_F,
\end{equation*}
which implies that
\[
\norm{\mathcal A(M)}^2=\norm{M}_F^2-(\langle A_1,M \rangle)^2 \geq \left(1-\frac{2r}{n}\right)\norm{M}_F^2.
\]
Define a scaled linear operator $\bar{\mathcal A}$ as
\[
\bar{\mathcal A}(M)=\sqrt\frac{n}{n-r}\mathcal A(M), \quad \forall M \in \RR^{n \times n}.
\]
Thus, the relation
\begin{equation}\label{eq:exmprip}
\left(1-\frac{r}{n-r}\right)\norm{M}_F^2 \leq \norm{\bar{\mathcal A}(M)}^2 \leq \left(1+\frac{r}{n-r}\right)\norm{M}_F^2
\end{equation}
holds for all $M \in \RR^{n \times n}$ with $\rank(M) \leq 2r$.

After choosing $A_1=(1/\sqrt n)I_n$ in the above argument, let $\mathcal A$ be the resulting linear operator and $\mathcal Q$ be the quadratic form in \eqref{eq:linearquad} that corresponds to the scaled linear operator $\bar{\mathcal A}$. By the same argument, a similar linear operator $\mathcal A'$ and the corresponding quadratic form $\mathcal Q'$ can be obtained after choosing
\begin{equation}\label{eq:a1prime}
A'_1=\frac{1}{\sqrt n}\diag(1,1,-1,-1,1,\dots,1).
\end{equation}
Now, we select $K=\diag(1,1,0,0,0,\dots,0)$ and $L=\diag(0,0,1,1,0,\dots,0)$. Then,
\begin{equation}\label{eq:exmpbdp}
\begin{aligned}
\abs{[\mathcal Q-\mathcal Q'](K,L)}&=\frac{n}{n-r}\abs{\langle \mathcal A(K),\mathcal A(L)\rangle-\langle \mathcal A'(K),\mathcal A'(L)\rangle} \\
&=\frac{n}{n-r}\abs{-\langle A_1,K\rangle\langle A_1,L\rangle+\langle A'_1,K\rangle\langle A'_1,L\rangle} \\
&=\frac{4}{n-r}\norm{K}_F\norm{L}_F.
\end{aligned}
\end{equation}
In the case $r=1$, it follows from \eqref{eq:exmprip} that both of the constructed quadratic forms $\mathcal Q$ and $\mathcal Q'$ satisfy $\delta$-$\RIP_2$ with $\delta=1/(n-1)$. If one can find a twice continuously differentiable function $f$ satisfying $\delta$-$\RIP_2$ such that
\[
\nabla^2f(M)=\mathcal Q, \quad \nabla^2f(M')=\mathcal Q'
\]
hold at two particular points $M,M' \in \RR^{n \times n}$ with $\rank(M) \leq 2$ and $\rank(M') \leq 2$, then by \eqref{eq:exmpbdp} the function $f$ cannot satisfy $\kappa$-$\BDP_2$ for $\kappa<4\delta$. Since the design of such function is cumbersome, we will use a weaker result that serves the same purpose. This result, to be formalized in Lemma~\ref{lem:extension}, states that for every $\mu>0$, one can find a twice continuously differentiable function $f$ with $(\delta+\mu)$-$\RIP_2$ and two matrices $M,M' \in \RR^{n \times n}$ of rank at most $1$ satisfying the following inequalities:
\begin{equation}\label{eq:approxhessian}
\begin{gathered}
\abs{[\nabla^2f(M)-\mathcal Q](K,L)} \leq \mu\norm{K}_F\norm{L}_F, \\
\abs{[\nabla^2f(M')-\mathcal Q'](K,L)} \leq \mu\norm{K}_F\norm{L}_F.
\end{gathered}
\end{equation}
Combining \eqref{eq:exmpbdp} and \eqref{eq:approxhessian} yields that
\[
\abs{[\nabla^2f(M)-\nabla^2f(M')](K,L)} \leq (4\delta+2\mu)\norm{K}_F\norm{L}_F.
\]
Therefore, the function $f$ cannot satisfy the $\kappa$-$\BDP_2$ property for any $\kappa<4\delta+2\mu$. Since $\mu$ can be made arbitrarily small, this shows that the constant $4\delta$ in Theorem~\ref{thm:rip2rbdp} cannot be improved. Similarly, by choosing $r=2$ instead of $r=1$ and repeating the above argument, one can show that the constant $2\delta$ in Theorem~\ref{thm:rip4rbdp} cannot be improved either.

\begin{lemma}\label{lem:extension}
Consider two quadratic forms $\mathcal Q$ and $\mathcal Q'$ satisfying the $\delta$-$\RIP_{2r}$ property. For every $\mu>0$, there exists a twice continuously differentiable function $f: \RR^{n \times n} \to \RR$ and two matrices $M,M' \in \RR^{n \times n}$ with $\rank(M) \leq 1$ and $\rank(M') \leq 1$ such that $f$ satisfies the $(\delta+\mu)$-$\RIP_{2r}$ property and that \eqref{eq:approxhessian} holds for all $K,L \in \RR^{n \times n}$.
\end{lemma}
\begin{proof}
Given $\mu>0$, let $f$ be given as
\[
f(V)=\frac{1}{2}[\mathcal Q'](V,V)+\frac{1}{2}H(\norm{V}_F^2)[\Delta](V,V),
\]
where $\Delta=\mathcal Q-\mathcal Q'$ and $H: \RR \to \RR$ is defined as
\[
H(t)=\begin{dcases*}
0, & if $t \leq 0$, \\
\exp(-1/t^\gamma), & if $t>0$.
\end{dcases*}
\]
Here, $\gamma \in (0,1)$ is a constant that will be determined later. It is straightforward to verify that $H$ is twice continuously differentiable and
\begin{subequations}
\begin{gather}
H'(0)=H''(0)=0, \label{eq:hzero} \\
\abs{tH'(t)} \leq \frac{\gamma}{\ee}, \quad \abs{t^2H''(t)} \leq \frac{4\gamma}{\ee}, \quad \forall t \in \RR. \label{eq:hderbound}
\end{gather}
\end{subequations}

The basic idea behind the above construction of $f$ is that when $\gamma$ is chosen to be small, the growth of the function $H$ becomes so slow that it can be regarded as a constant when computing the Hessian of the above function $f$. As a result, the Hessian is approximately a linear combination of two quadratic forms $\mathcal Q$ and $\mathcal Q'$ with the $\delta$-$\RIP_{2r}$ property. Formally, the Hessian $\nabla^2f(V)$ of $f$ at a particular matrix $V \in \RR^{n \times n}$, when applied to arbitrary $K,L \in \RR^{n \times n}$, is given by
\begin{equation}\label{eq:fhessian}
\begin{aligned}
[\nabla^2f(V)](K,L)&=2H''(\norm{V}_F^2)[\Delta](V,V)\langle V,K\rangle\langle V,L\rangle+H'(\norm{V}_F^2)[\Delta](V,V)\langle K,L\rangle \\
&+2H'(\norm{V}_F^2)([\Delta](L,V)\langle V,K\rangle+[\Delta](K,V)\langle V,L\rangle) \\
&+[\mathcal Q'+H(\norm{V}_F^2)\Delta](K,L).
\end{aligned}
\end{equation}
By compactness, there exists a constant $C>0$ such that
\begin{equation}\label{eq:deltabound}
\abs{[\Delta](A,B)} \leq C\norm{A}_F\norm{B}_F
\end{equation}
holds for all $A,B \in \RR^{n \times n}$. We choose a sufficiently small $\gamma$ such that $26\gamma C/\ee \leq \mu$. By \eqref{eq:hderbound}, \eqref{eq:fhessian}, \eqref{eq:deltabound} and the Cauchy-Schwartz inequality, we have
\begin{equation}\label{eq:hessianbound}
\abs{[\nabla^2f(V)-\mathcal Q'-H(\norm{V}_F^2)\Delta](K,L)} \leq \frac{13\gamma C}{\ee}\norm{K}_F\norm{L}_F \leq \frac{\mu}{2}\norm{K}_F\norm{L}_F.
\end{equation}

To prove that the function $f$ satisfies $(\delta+\mu)$-$\RIP_{2r}$, assume for now that $K=L$ and $\rank(K) \leq 2r$.
The inequality $0 \leq H(\norm{V}_F^2) \leq 1$ and the $\delta$-$\RIP_{2r}$ property of $\mathcal Q$ and $\mathcal Q'$ imply that
\begin{equation*}
(1-\delta)\norm{K}_F^2 \leq [\mathcal Q'+H(\norm{V}_F^2)\Delta](K,K) \leq (1+\delta)\norm{K}_F^2.
\end{equation*}
By \eqref{eq:hessianbound} and the above inequality, the function $f$ satisfies the $(\delta+\mu)$-$\RIP_{2r}$ property. To prove the existence of $M$ and $M'$ satisfying \eqref{eq:approxhessian}, we select $M'=0$ and
\[
M=\diag(s,0,\dots,0).
\]
For any $K,L \in \RR^{n \times n}$, it follows from \eqref{eq:hzero} and \eqref{eq:fhessian} that
\begin{equation}\label{eq:exacthessian}
[\nabla^2f(M')-\mathcal Q'](K,L)=0.
\end{equation}
Moreover, \eqref{eq:deltabound} and \eqref{eq:hessianbound} yield that
\begin{align*}
\abs{[\nabla^2f(M)-\mathcal Q](K,L)} &\leq \frac{\mu}{2}\norm{K}_F\norm{L}_F+\abs{[\mathcal Q'+H(\norm{M}_F^2)\Delta-\mathcal Q](K,L)} \\
&\leq \left(\frac{\mu}{2}+(1-H(\norm{M}_F^2))C\right)\norm{K}_F\norm{L}_F.
\end{align*}
Since $H(\norm{M}_F^2) \to 1$ as $s \to +\infty$, \eqref{eq:approxhessian} is satisfied as long as $s$ is sufficiently large.
\end{proof}

The above argument also provides examples of the function $f$ whose corresponding recovery problem \eqref{eq:nonlinearrec} can be certified to have no spurious local minima via Theorem~\ref{thm:main}, while the existing results in the literature fail to do so. Following the above construction, choose $n=4$, $r=1$, and let
\[
\tilde f(V)=\frac{1-\lambda}{2}[\mathcal Q'](V,V)+\lambda f(V),
\]
for some $\lambda \in [0,1]$. The Hessian can be written as
\begin{equation}\label{eq:tildefhessian}
\nabla^2\tilde f(V)=(1-\lambda)\mathcal Q'+\lambda\nabla^2f(V).
\end{equation}
If $\lambda>0$, the Hessian of $\tilde f$ is not a constant, and therefore the condition in \cite{ZSL2019} cannot be applied. On the other hand, it follows from \eqref{eq:exacthessian} that
\[
[\nabla^2\tilde f(0)](A_1',A_1')=[\mathcal Q'](A_1',A_1')=0,
\]
for the matrix $A_1'$ of rank $4$ defined in \eqref{eq:a1prime}. Thus, the function $\tilde f$ cannot satisfy the $\delta$-$\RIP_{2,4}$ property for any $\delta \in [0,1)$. This implies that the condition in \cite{LZT2019} cannot be applied either. In contrast, note that the quadratic form $\mathcal Q'$ satisfies the $1/3$-$\RIP_2$ property and the function $f$ satisfies the $(1/3+\mu)$-$\RIP_2$ property. Therefore, it can be concluded from \eqref{eq:tildefhessian} that the function $\tilde f$ also satisfies the $(1/3+\mu)$-$\RIP_2$ property. In light of Theorem~\ref{thm:rip2rbdp}, $f$ satisfies $4(1/3+\mu)$-$\BDP_2$ and thus $f'$ satisfies $4\lambda(1/3+\mu)$-$\BDP_2$. Hence, Theorem~\ref{thm:main} certifies the absence of spurious local minima as long as $\lambda$ and $\mu$ jointly satisfy
\[
\frac{1}{3}+\mu<\frac{2-6(1+\sqrt 2)4\lambda(1/3+\mu)}{4+6(1+\sqrt 2)4\lambda(1/3+\mu)}.
\]

\section{Proofs of Main Results}\label{sec:nospurious}

Our approach consists of two major steps. The first step is to find necessary conditions that the function $f$ with the $\delta$-$\RIP_{2r}$ and $\kappa$-$\BDP_{2r}$ properties must satisfy if the corresponding problem \eqref{eq:nonlinearrec} has a local minimizer $X$ such that $XX^T \neq M^*$, where $M^*$ is the ground truth. The second step is to develop certain conditions on $\delta$ and $\kappa$ that rule out the satisfaction of the above necessary condition.

Before proceeding with the proofs, we need to introduce some notations. Given two matrices $X,Z \in \RR^{n \times r}$, define
\[
\mathbf e=\vect(XX^T-ZZ^T) \in \RR^{n^2},
\]
and let $\mathbf X \in \RR^{n^2 \times nr}$ be the matrix satisfying
\[
\mathbf X\vect U=\vect(XU^T+UX^T), \quad \forall U \in \RR^{n \times r}.
\]

\subsection{Necessary Conditions for the Existence of Spurious Local Minima}

As the first step, in the following lemma we obtain necessary conditions for the existence of spurious local minima in the problem \eqref{eq:nonlinearrec}.

\begin{lemma}\label{lem:necessary}
Assume that the function $f$ in the problem \eqref{eq:nonlinearrec} satisfies the $\delta$-$\RIP_{2r}$ and $\kappa$-$\BDP_{2r}$ properties. If $X$ is a local minimizer of \eqref{eq:nonlinearrec} and $Z$ is a global minimizer of \eqref{eq:nonlinearrec} with $M^*=ZZ^T$, then there exists a symmetric matrix $\mathbf H \in \RR^{n^2 \times n^2}$ such that the following three conditions hold:
\begin{enumerate}
\item $\norm{\mathbf X^T\mathbf H\mathbf e} \leq 2\kappa\sqrt{\lambda_1(XX^T)}\norm{\mathbf e}$;
\item $2I_r \otimes \mat(\mathbf H\mathbf e)+\mathbf X^T\mathbf H\mathbf X \succeq -2\kappa\norm{\mathbf e}I_{nr}$;
\item $\mathbf H$ satisfies the $\delta$-$\RIP_{2r}$ property, i.e, for every matrix $U \in \RR^{n \times n}$ with $\rank(U) \leq 2r$, it holds that
\[
(1-\delta)\norm{\mathbf U}^2 \leq \mathbf U^T\mathbf H\mathbf U \leq (1+\delta)\norm{\mathbf U}^2,
\]
where $\mathbf U=\vect U$.
\end{enumerate}
\end{lemma}
\begin{proof}
Choose $\mathbf H$ to be the matrix satisfying
\[
(\vect K)^T\mathbf H\vect L=[\nabla^2f(XX^T)](K,L),
\]
for all $K,L \in \RR^{n \times n}$. Condition 3 follows immediately from the $\delta$-$\RIP_{2r}$ property of the function $f$. To prove the remaining two conditions, define $g(Y)=f(YY^T)$ and $M=XX^T$. Since $X$ is a local minimizer of the function $g(\cdot)$, for every $U \in \RR^{n \times r}$ with $\mathbf U=\vect U$, the first-order optimality condition implies that
\begin{equation}\label{eq:firstordercond}
0=\langle \nabla g(X),U \rangle=\langle \nabla f(M), XU^T+UX^T \rangle.
\end{equation}
Define an auxiliary function $h:\RR^{n \times n} \to \RR$ by letting
\[
h(V)=\langle \nabla f(V), XU^T+UX^T \rangle.
\]
By the mean value theorem, there exists a matrix $\xi$ on the segment between $M$ and $M^*$ such that
\begin{equation}\label{eq:xizero}
[\nabla^2f(\xi)](M-M^*,XU^T+UX^T)=\langle \nabla h(\xi),M-M^* \rangle=h(M)-h(M^*)=0,
\end{equation}
in which the last equality follows from \eqref{eq:firstordercond} and $\nabla f(M^*)=0$. Since $\rank(M) \leq r$ and $\rank(M^*) \leq r$, we have $\rank(\xi) \leq 2r$ and $\rank(M-M^*) \leq 2r$. Applying the $\kappa$-$\BDP_{2r}$ property to the Hessian of $f(\cdot)$ at matrices $M$ and $\xi$, together with \eqref{eq:xizero}, one can obtain
\begin{align*}
\abs{\mathbf e^T\mathbf H\mathbf X\mathbf U}&=\abs{[\nabla^2f(M)](M-M^*,XU^T+UX^T)} \\
&\leq \kappa\norm{M-M^*}_F\norm{XU^T+UX^T}_F \\
&\leq 2\kappa\norm{\mathbf e}\norm{XU^T}_F \\
&=2\kappa\norm{\mathbf e}\sqrt{\tr(UX^TXU^T)} \\
&\leq 2\kappa\norm{\mathbf e}\sqrt{\lambda_1(XX^T)}\norm{\mathbf U}
\end{align*}
Condition 1 can be proved by setting $\mathbf U=\mathbf X^T\mathbf H\mathbf e$.

For every $U \in \RR^{n \times r}$ with $\mathbf U=\vect U$, the second-order optimality condition gives
\begin{equation}\label{eq:hessianeq1}
0 \leq [\nabla^2g(X)](U,U)=[\nabla^2f(M)](XU^T+UX^T,XU^T+UX^T)+2\langle \nabla f(M),UU^T \rangle.
\end{equation}
The first term on the right-hand side can be equivalently written as $(\mathbf X\mathbf U)^T\mathbf H(\mathbf X\mathbf U)$. A similar argument can be made to conclude that there exists another matrix $\xi'$ on the segment between $M$ and $M^*$ such that
\begin{align}
\langle \nabla f(M),UU^T \rangle&=\langle \nabla f(M)-\nabla f(M^*),UU^T \rangle \nonumber \\
&=[\nabla^2f(\xi')](M-M^*,UU^T) \nonumber \\
&\leq [\nabla^2f(M)](M-M^*,UU^T) \nonumber+\kappa\norm{M-M^*}_F\norm{UU^T}_F \nonumber \\
&=\vect(UU^T)\mathbf H\mathbf e+\kappa\norm{\mathbf e}\norm{\mathbf U}^2 \nonumber \\
&=\frac{1}{2}(\vect U)^T\vect((W+W^T)U)+\kappa\norm{\mathbf e}\norm{\mathbf U}^2 \nonumber \\
&=\mathbf U^T(I_r \otimes \mat(\mathbf H\mathbf e))\mathbf U+\kappa\norm{\mathbf e}\norm{\mathbf U}^2, \label{eq:hessianeq2}
\end{align}
in which $W \in \RR^{n \times n}$ is the unique matrix satisfying $\vect W=\mathbf H\mathbf e$. Condition 2 can be obtained by combining \eqref{eq:hessianeq1} and \eqref{eq:hessianeq2}.
\end{proof}

For given $X,Z \in \RR^{n \times r}$ and $\kappa \geq 0$, one can construct an optimization problem based on the conditions in Lemma~\ref{lem:necessary} as follows:
\begin{equation}\label{eq:ripopt}
\begin{aligned}
\min_{\delta,\mathbf H} \quad & \delta \\
\st \quad & \norm{\mathbf X^T\mathbf H\mathbf e} \leq a, \\
& 2I_r \otimes \mat(\mathbf H\mathbf e)+\mathbf X^T\mathbf H\mathbf X \succeq -bI_{nr}, \\
& \text{$\mathbf H$ is symmetric and satisfies $\delta$-$\RIP_{2r}$},
\end{aligned}
\end{equation}
where
\begin{equation}\label{eq:abdef}
a=2\kappa\sqrt{\lambda_1(XX^T)}\norm{\mathbf e}, \quad b=2\kappa\norm{\mathbf e}.
\end{equation}
Let $\delta(X,Z;\kappa)$ be the optimal value of \eqref{eq:ripopt}. Assume that $f$ in the original problem \eqref{eq:nonlinearrec} satisfies $\delta$-$\RIP_{2r}$ and $\kappa$-$\BDP_{2r}$. By Lemma~\ref{lem:necessary}, if $X$ is a local minimizer of \eqref{eq:nonlinearrec} and $Z$ is a global minimizer of \eqref{eq:nonlinearrec} with $M^*=ZZ^T$, then $\delta \geq \delta(X,Z;\kappa)$. As a result, by defining $\delta^*(\kappa)$ as the optimal value of the optimization problem
\[
\min_{X,Z \in \RR^{n \times r}}\delta(X,Z;\kappa) \quad \st \quad XX^T \neq ZZ^T,
\]
the problem \eqref{eq:nonlinearrec} is guaranteed to have no spurious local minima as long as $\delta<\delta^*(\kappa)$.

The remaining task is to compute $\delta(X,Z;\kappa)$ and $\delta^*(\kappa)$. First, by the property of the Schur complement, the first constraint in \eqref{eq:ripopt} can be equivalently written as
\[
\begin{bmatrix}
I_{nr} & \mathbf X^T\mathbf H\mathbf e \\
(\mathbf X^T\mathbf H\mathbf e)^T & a^2
\end{bmatrix} \succeq 0.
\]
The major difficulty of solving \eqref{eq:ripopt} comes from the last constraint, since it is NP-hard to verify whether a given quadratic form satisfies $\delta$-$\RIP_{2r}$ \citep{TP2014}. Instead, we tighten the last constraint of \eqref{eq:ripopt} by requiring $\mathbf H$ to have a norm-preserving property for all matrices instead of just for matrices with rank at most $2r$, i.e.,
\[
(1-\delta)\norm{\mathbf U}^2 \leq \mathbf U^T\mathbf H\mathbf U \leq (1+\delta)\norm{\mathbf U}^2, \quad \forall \mathbf U \in \RR^{n^2},
\]
which leads to following semidefinite program:
\begin{equation}\label{eq:ripsdp}
\begin{aligned}
\min_{\delta,\mathbf H} \quad & \delta \\
\st \quad & \begin{bmatrix}
I_{nr} & \mathbf X^T\mathbf H\mathbf e \\
(\mathbf X^T\mathbf H\mathbf e)^T & a^2
\end{bmatrix} \succeq 0, \\
& 2I_r \otimes \mat(\mathbf H\mathbf e)+\mathbf X^T\mathbf H\mathbf X \succeq -bI_{nr}, \\
& (1-\delta)I_{n^2} \preceq \mathbf H \preceq (1+\delta)I_{n^2}.
\end{aligned}
\end{equation}

Similar to the case with linear measurements studied in \cite{ZSL2019}, due to the symmetry under orthogonal projections, the problems \eqref{eq:ripopt} and \eqref{eq:ripsdp} turn out to have the same optimal value. This result is a direct generalization of \cite[Theorem~8]{ZSL2019} to the case with nonlinear measurements. See Appendix~\ref{sec:app2} for the proof.

\begin{lemma}\label{lem:relax}
For given $X,Z \in \RR^{n \times r}$ and $\kappa \geq 0$, the optimization problems \eqref{eq:ripopt} and \eqref{eq:ripsdp} have the same optimal value.
\end{lemma}

Even if the value of $\delta(X,Z;\kappa)$ for given $X$, $Z$ and $\kappa$ can now be efficiently calculated by solving the semidefinite program \eqref{eq:ripsdp}, to further compute $\delta^*(\kappa)$, an analytical expression is still needed for $\delta(X,Z;\kappa)$. For our purpose, it is sufficient to find a lower bound on $\delta(X,Z;\kappa)$. In the remainder of this section, as a last step to prove Theorem~\ref{thm:main} and Theorem~\ref{thm:local}, we will focus on the problem of lower bounding $\delta(X,Z;\kappa)$ and $\delta^*(\kappa)$. Before digging into this problem, we shall first study the function $\delta(X,Z;\kappa)$ numerically to gain some intuition.

\subsection{Numerical Illustration}

To numerically analyze $\delta(X,Z;\kappa)$, we select different values for $\kappa$, and in each case we sample $X$ and $Z$ randomly by drawing each entry of these matrices independently from the standard normal distribution and then solve the semidefinite program \eqref{eq:ripsdp} to evaluate $\delta(X,Z;\kappa)$. The empirical cumulative distributions of $\delta(X,Z;\kappa)$ from 5000 samples for $n=5$ and different $\kappa$ and rank $r$ are given in Figure~\ref{fig:delta}. It can be observed that when $\kappa$ increases, $\delta(X,Z;\kappa)$ becomes smaller and a worse bound is expected from finding the minimum value $\delta^*(\kappa)$ of $\delta(X,Z;\kappa)$. On the other hand, $\delta(X,Z;\kappa)$ increases when $r$ grows. For example, in the case when $\kappa=0.05$, all of the samples satisfy $\delta(X,Z;\kappa) \geq 0.44$ for rank $r=1$, $\delta(X,Z;\kappa) \geq 0.51$ for rank $r=2$, and $\delta(X,Z;\kappa) \geq 0.64$ for rank $r=3$. This observation suggests that the stochastic gradient method may perform better in the higher-rank cases, since its trajectory during the iteration is less likely to be close to a spurious local minimizer, and it will be easier to escape even if the trajectory encounters a spurious solution $X$ that is not detected in the above sampling process.

\subsection{Global Guarantee for the Rank-1 Case}

When $r=1$, $X$ and $Z$ reduce to vectors and henceforth will be denoted as $x$ and $z$ with
\[
\mathbf e=x \otimes x-z \otimes z, \quad \mathbf Xu=x \otimes u+u \otimes x, \quad \sqrt{\lambda_1(xx^T)}=\norm{x}.
\]
Moreover,
\begin{equation}\label{eq:xsvd}
\norm{\mathbf Xu}^2=2\norm{x}^2\norm{u}^2+2(x^Tu)^2, \quad \forall u \in \RR^n.
\end{equation}

Given two vectors $x,z \in \RR^n$ with $x \neq 0$ and $xx^T \neq zz^T$, one can find a unit vector $w \in \RR^n$ such that $w$ is orthogonal to $x$ and $z=c_1x+c_2w$ for some scalars $c_1$ and $c_2$. Then,
\[
\mathbf e=\mathbf X\tilde y-c_2^2(w \otimes w),
\]
in which
\[
\tilde y=\frac{1-c_1^2}{2}x-c_1c_2w.
\]
Note that $\mathbf X\tilde y$ is orthogonal to $w \otimes w$. Furthermore, since $\tilde y \neq 0$ by $xx^T \neq zz^T$ and thus $\mathbf X\tilde y \neq 0$ by \eqref{eq:xsvd}, one can rescale $\tilde y$ into $\hat y$ such that $\norm{\mathbf X\hat y}=1$ and
\begin{equation}\label{eq:error}
\mathbf e=\norm{\mathbf e}(\sqrt{1-\alpha^2}\mathbf X\hat y-\alpha(w \otimes w)),
\end{equation}
with
\begin{equation}\label{eq:alpha}
\alpha:=\frac{c_2^2}{\norm{\mathbf e}}=\frac{\norm{z}^2-(x^Tz/\norm{x})^2}{\norm{\mathbf e}}.
\end{equation}
In addition, \eqref{eq:xsvd} also implies
\begin{equation}\label{eq:ynorm}
\norm{\hat y} \leq \frac{\norm{\mathbf X\hat y}}{\sqrt 2\norm{x}}=\frac{1}{\sqrt 2\norm{x}}.
\end{equation}

\begin{figure}[!t]
\centering
\subfloat[$\kappa=0$]{\includegraphics[width=6cm]{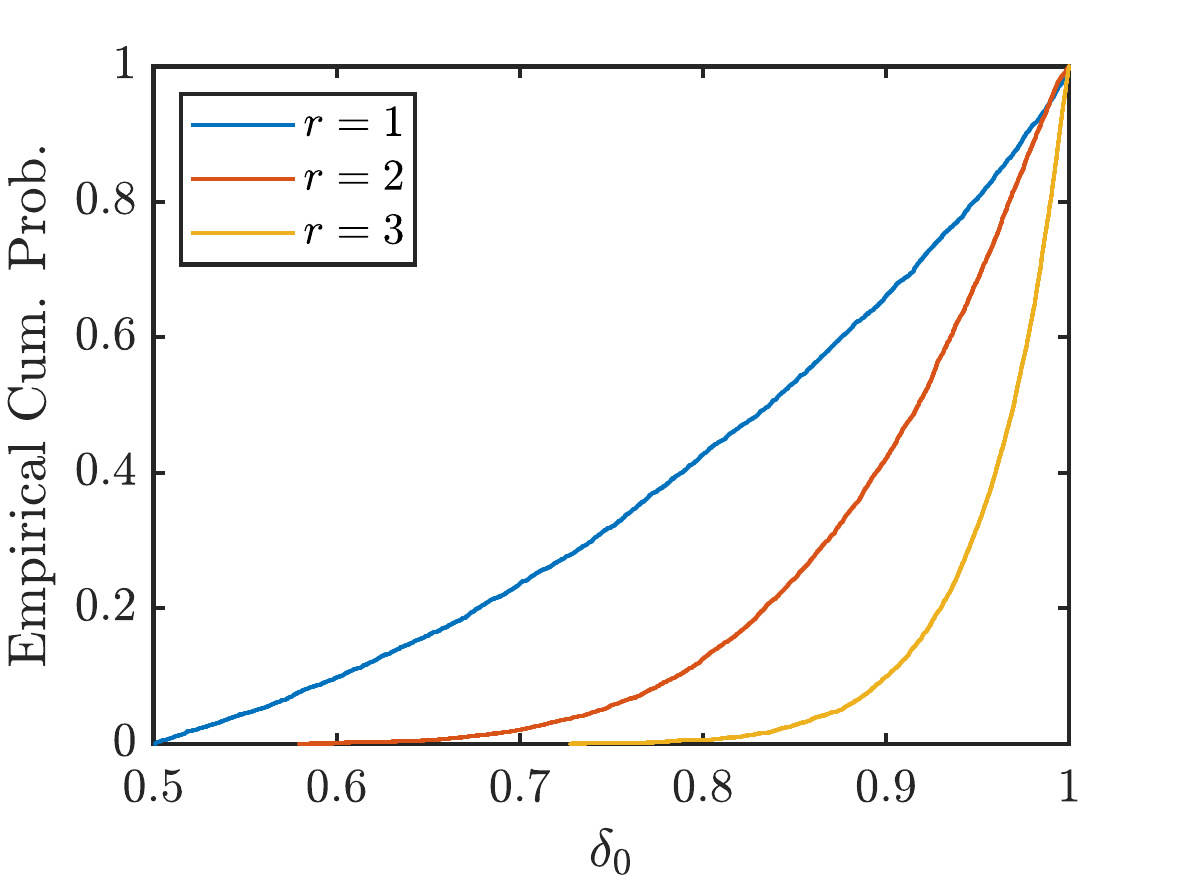}}\quad
\subfloat[$\kappa=0.05$]{\includegraphics[width=6cm]{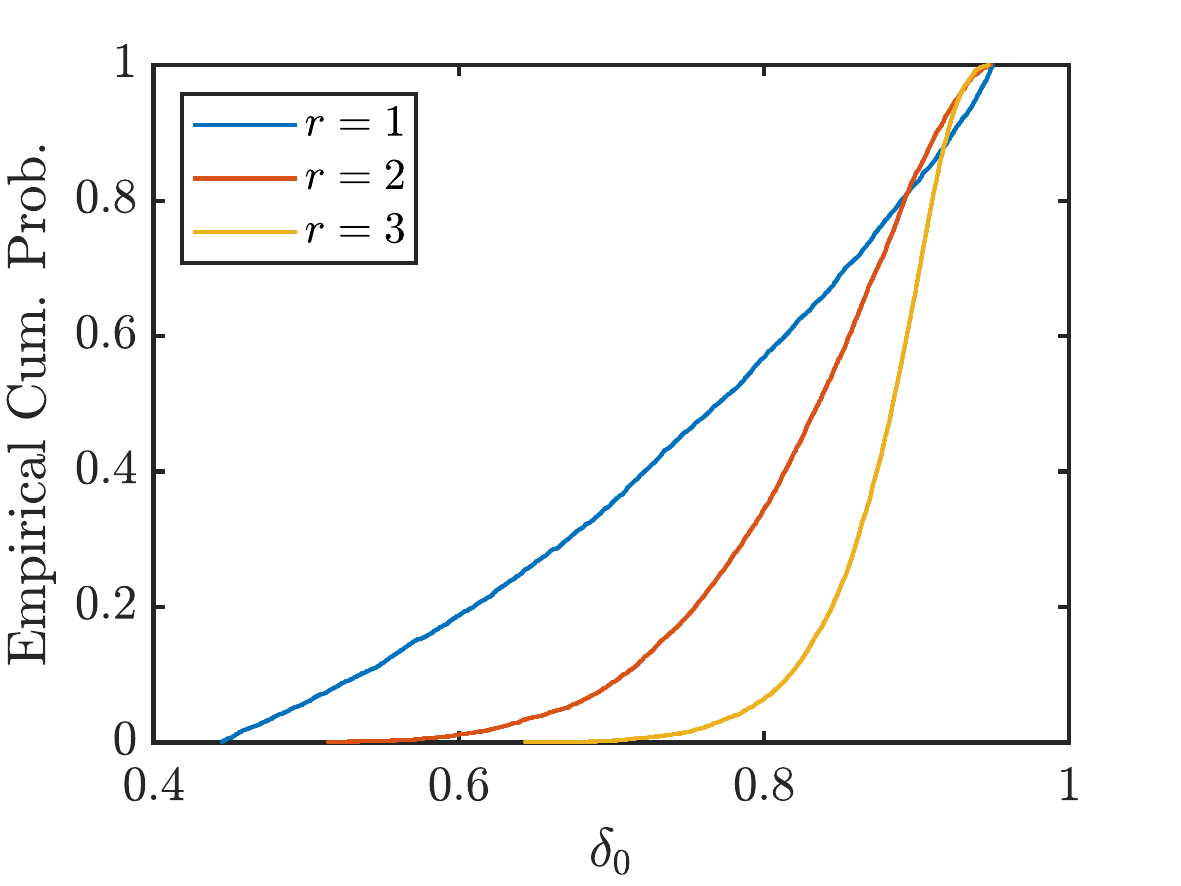}}\\
\subfloat[$\kappa=0.1$]{\includegraphics[width=6cm]{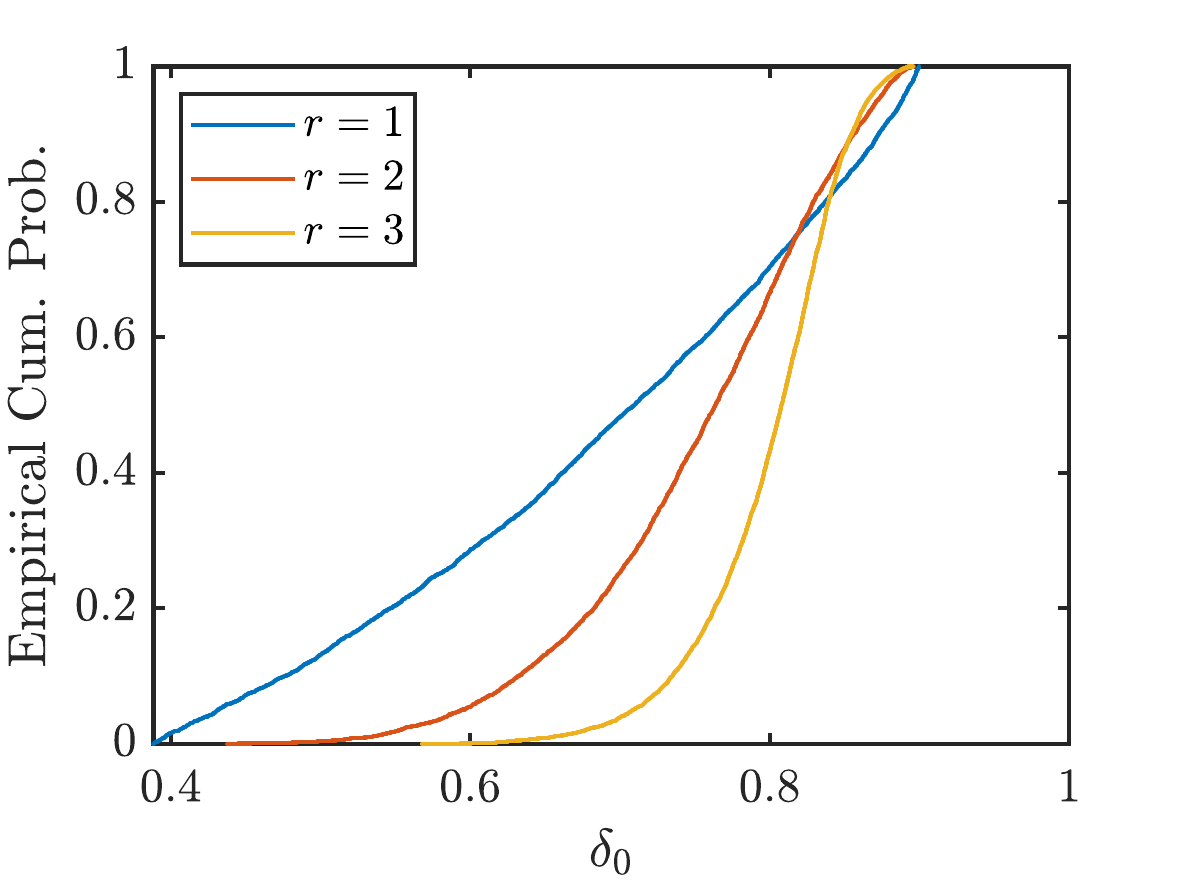}}\quad
\subfloat[$\kappa=0.2$]{\includegraphics[width=6cm]{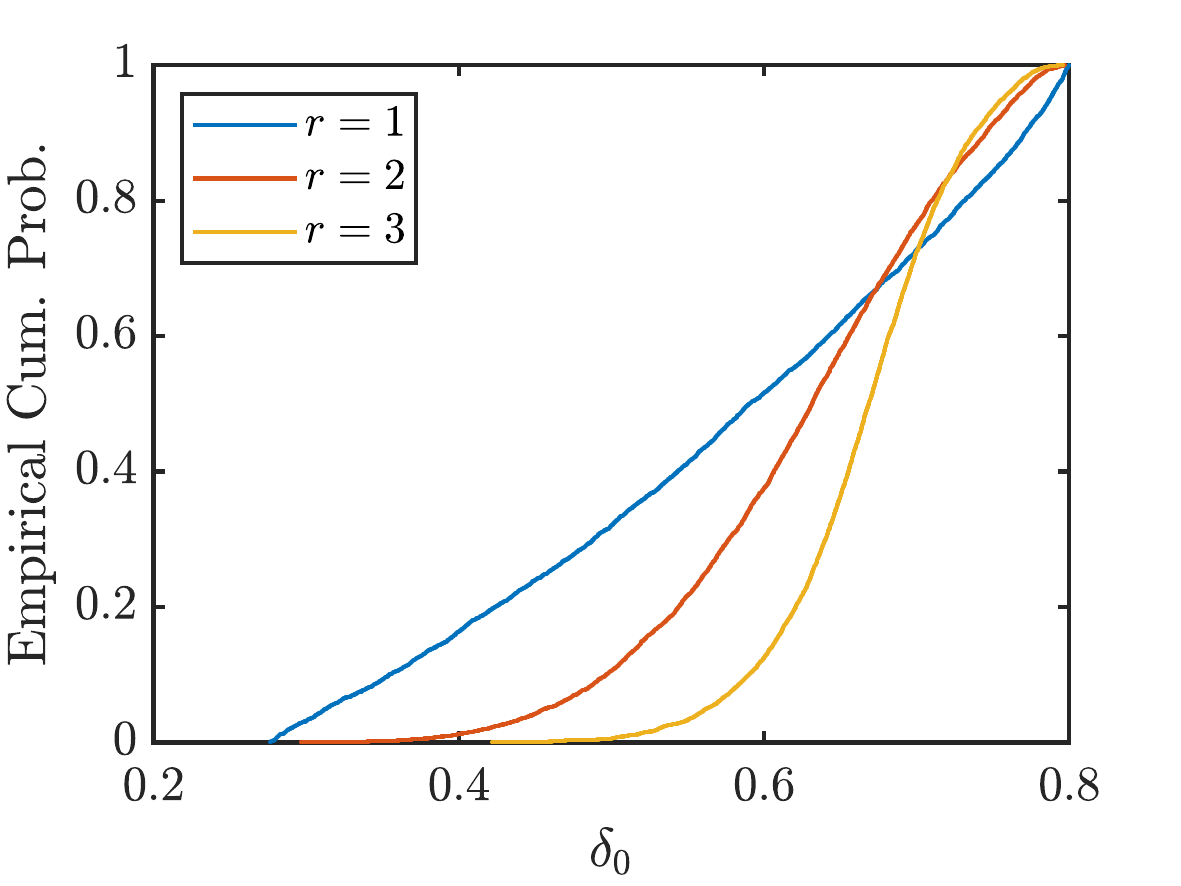}}
\caption{The empirical probability of $\delta(X,Z;\kappa) \leq \delta_0$ for randomly generated $X$ and $Z$ matrices with $n=5$ and different $\kappa$ and rank $r$.}\label{fig:delta}
\end{figure}

To proceed with our proof, we will need the next lemma that studies the eigenvalues of some structured rank-2 matrices.

\begin{lemma}[\cite{ZSL2019}]\label{lem:eigangle}
Let $u$ and $v$ be two vectors of the same dimension. The eigenvalues of the matrix $uv^T+vu^T$ can take only three possible values
\[
\norm{u}\norm{v}(1+\cos\theta), \quad -\norm{u}\norm{v}(1-\cos\theta), \quad 0,
\]
where $\theta$ is the angle between $u$ and $v$.
\end{lemma}

\begin{lemma}\label{lem:deltamin}
Let $x,z \in \RR^n$ with $xx^T \neq zz^T$. The optimal value $\delta(x,z;\kappa)$ of \eqref{eq:ripsdp} satisfies
\[
\delta(x,z;\kappa) \geq \frac{1-\eta_0(x,z)-2(1+\sqrt 2)\kappa}{1+\eta_0(x,z)+2(1+\sqrt 2)\kappa},
\]
in which
\[
\eta_0(x,z)=\begin{dcases*}
\frac{1-\sqrt{1-\alpha^2}}{1+\sqrt{1-\alpha^2}}, & if $\beta \geq \dfrac{\alpha}{1+\sqrt{1-\alpha^2}}$, \\
\frac{\beta(\beta-\alpha)}{\beta\alpha-1}, & if $\beta \leq \dfrac{\alpha}{1+\sqrt{1-\alpha^2}}$,
\end{dcases*}
\]
with $\alpha$ defined in \eqref{eq:alpha}\footnote{When $x=0$, $\alpha$ is defined to be $\norm{z}^2/\norm{\mathbf e}$.} and $\beta=\norm{x}^2/\norm{\mathbf e}$.
\end{lemma}
\begin{proof}
Define $\eta(x,z;\kappa)$ to be the optimal value of the following optimization problem:
\begin{equation}\label{eq:etaopt}
\begin{aligned}
\max_{\eta,\mathbf H} \quad & \eta \\
\st \quad & \begin{bmatrix}
I_{nr} & \mathbf X^T\mathbf H\mathbf e \\
(\mathbf X^T\mathbf H\mathbf e)^T & a^2
\end{bmatrix} \succeq 0, \\
& 2\mat(\mathbf H\mathbf e)+\mathbf X^T\mathbf X \succeq -bI_{nr}, \\
& \eta I_{n^2} \preceq \mathbf H \preceq I_{n^2}.
\end{aligned}
\end{equation}
It can be verified that
\begin{equation}\label{eq:deltaeta}
\eta(x,z;\kappa) \geq \frac{1-\delta(x,z;\kappa)}{1+\delta(x,z;\kappa)},
\end{equation}
because given any feasible solution $(\delta,\mathbf H)$ to \eqref{eq:ripsdp}, the point
\[
\left(\frac{1-\delta}{1+\delta},\frac{1}{1+\delta}\mathbf H\right)
\]
is also a feasible solution to \eqref{eq:etaopt}. The reason is that the first and last constraints in \eqref{eq:etaopt} naturally hold while the second constraint is satisfied due to
\begin{equation*}
2\mat\left(\frac{1}{1+\delta}\mathbf H\mathbf e\right)+\mathbf X^T\mathbf X \succeq
\frac{1}{1+\delta}(2\mat(\mathbf H\mathbf e)+\mathbf X^T\mathbf H\mathbf X) \succeq -\frac{b}{1+\delta}I_{nr} \succeq -bI_{nr}.
\end{equation*}
Therefore, to find a lower bound on $\delta(x,z;\kappa)$, we only need to find an upper bound on $\eta(x,z;\kappa)$.

The dual problem of \eqref{eq:etaopt} can be written as
\begin{equation}\label{eq:etadual}
\begin{aligned}
\min_{U_1,U_2,V,G,\lambda,y} \quad & \tr(U_2)+\langle \mathbf X^T\mathbf X+bI_n,V\rangle+a^2\lambda+\tr(G), \\
\st \quad & \tr(U_1)=1, \\
& (\mathbf Xy-v)\mathbf e^T+\mathbf e(\mathbf Xy-v)^T=U_1-U_2, \\
& \begin{bmatrix}
G & -y \\
-y^T & \lambda
\end{bmatrix} \succeq 0, \\
& U_1 \succeq 0, \quad U_2 \succeq 0, \quad V \succeq 0, \quad v=\vect V.
\end{aligned}
\end{equation}
By weak duality, the dual objective value associated with any feasible solution to the dual problem \eqref{eq:etadual} is an upper bound on $\eta(x,z;\kappa)$.

In the case when $x \neq 0$, we fix a constant $\gamma \in [0,\alpha]$ and choose
\[
y=\frac{\sqrt{1-\gamma^2}}{\norm{\mathbf e}}\hat y, \quad v=\frac{\gamma}{\norm{\mathbf e}}(w \otimes w),
\]
where $\hat y$ and $w$ are the vectors defined before \eqref{eq:error}. Since $\norm{\mathbf X\hat y}=1$, $\norm{w \otimes w}=1$ and $\mathbf X\hat y$ is orthogonal to $w \otimes w$, it holds that
\[
\norm{\mathbf Xy-v}=\frac{1}{\norm{\mathbf e}}.
\]
Combined with \eqref{eq:error}, one can obtain
\[
\mathbf e^T(\mathbf Xy-v)=\psi(\gamma),
\]
where $\psi(\gamma)$ is given by
\[
\psi(\gamma)=\gamma\alpha+\sqrt{1-\gamma^2}\sqrt{1-\alpha^2}.
\]
Now, define
\[
M=(\mathbf Xy-v)\mathbf e^T+\mathbf e(\mathbf Xy-v)^T
\]
and decompose
\[
M=[M]_+-[M]_-,
\]
in which both $[M]_+ \succeq 0$ and $[M]_-\succeq 0$. Let $\theta$ be the angle between $\mathbf e$ and $\mathbf Xy-v$. By Lemma~\ref{lem:eigangle}, it holds that
\begin{align*}
\tr([M]_+)&=\norm{\mathbf e}\norm{\mathbf Xy-v}(1+\cos\theta)=1+\psi(\gamma), \\
\tr([M]_-)&=\norm{\mathbf e}\norm{\mathbf Xy-v}(1-\cos\theta)=1-\psi(\gamma).
\end{align*}
Then, it is routine to verify that
\begin{gather*}
U_1^*=\frac{[M]_+}{\tr([M]_+)}, \quad U_2^*=\frac{[M]_-}{\tr([M]_+)}, \\
v^*=\frac{v}{\tr([M]_+)}, \quad G^*=\frac{1}{\lambda^*}y^*y^{*T} \\
\lambda^*=\frac{\norm{y^*}}{a}, \quad y^*=\frac{y}{\tr([M]_+)}
\end{gather*}
forms a feasible solution to the dual problem \eqref{eq:etadual} whose objective value is equal to
\begin{equation}\label{eq:obj}
\frac{\tr([M]_-)+\langle \mathbf X^T\mathbf X+bI_n,V\rangle+2a\norm{y}}{\tr([M]_+)}.
\end{equation}
By \eqref{eq:xsvd} and \eqref{eq:ynorm}, one can write
\begin{gather}
\langle \mathbf X^T\mathbf X+bI_n,V\rangle=\frac{\gamma}{\norm{\mathbf e}}(\norm{\mathbf Xw}^2+b)=\frac{\gamma}{\norm{\mathbf e}}(2\norm{x}^2+b)=2(\beta+\kappa)\gamma, \label{eq:obj1} \\
2a\norm{y} \leq \frac{2a\norm{\hat y}}{\norm{\mathbf e}} \leq 2\sqrt 2\kappa, \label{eq:obj2}
\end{gather}
where $a$ and $b$ are defined in \eqref{eq:abdef}. Substituting \eqref{eq:obj1} and \eqref{eq:obj2} into \eqref{eq:obj} yields that
\[
\eta(x,z;\kappa) \leq \Psi(\gamma)+2(1+\sqrt{2})\kappa,
\]
where
\[
\Psi(\gamma)=\frac{2\beta\gamma+1-\psi(\gamma)}{1+\psi(\gamma)}.
\]
A simple calculation shows that the function $\Psi(\gamma)$ has at most one stationary point over the interval $(0,\alpha)$ and
\[
\min_{0 \leq \gamma \leq \alpha}\Psi(\gamma)=\eta_0(x,z).
\]
In the case when $x=0$, we have $\eta_0(x,z)=0$, and
\begin{gather*}
U_1=\frac{\mathbf e\mathbf e^T}{\norm{\mathbf e}^2}, \quad U_2=0, \quad V=\frac{zz^T}{2\norm{\mathbf e}^2}, \\
G=0, \quad \lambda=0, \quad y=0
\end{gather*}
forms a feasible solution to the dual problem \eqref{eq:etadual}, which implies that
\[
\eta(x,z,\kappa) \leq \langle bI_n,V \rangle=\kappa.
\]
In either case, it holds that
\[
\eta(x,z;\kappa) \leq \eta_0(x,z)+2(1+\sqrt{2})\kappa,
\]
which gives the desired result after combining it with \eqref{eq:deltaeta}.
\end{proof}

\par\noindent{\bf Proof of Theorem~\ref{thm:main}\ }
By Lemma~\ref{lem:relax} and the discussion after Lemma~\ref{lem:necessary}, we only need to show that
\begin{equation}\label{eq:boundproof}
\delta(x,z;\kappa) \geq \frac{2-6(1+\sqrt 2)\kappa}{4+6(1+\sqrt 2)\kappa},
\end{equation}
for all $x,z \in \RR^n$ with $xx^T \neq zz^T$. Similarly to the approach used in proof of \cite[Theorem~3]{ZSL2019}, it can be verified that the function $\eta_0(x,z)$ defined in the statement of Lemma~\ref{lem:deltamin} has the maximum value $1/3$ that is attained by any two vectors $x$ and $z$ that are orthogonal to each other such that $\norm{x}/\norm{z}=1/2$. Consequently, \eqref{eq:boundproof} holds in light of Lemma~\ref{lem:deltamin}.
\hfill\BlackBox\\[2mm]

\subsection{Local Guarantee for the Rank-r Case}

The key step in the proof of Theorem~\ref{thm:main} is to derive a closed-form expression serving as a lower bound on $\delta(X,Z;\kappa)$. Similar to the idea used in \cite{ZZ2020}, we need to first simplify $\delta(X,Z;\kappa)$ for the higher-rank cases by removing the constraint corresponding to the second-order optimality condition in \eqref{eq:ripsdp}. The first step is to establish the following lemma that is similar to Lemma~\ref{lem:necessary} but ignores the second-order condition.

\begin{lemma}\label{lem:necessaryfoc}
Assume that the function $f$ in the problem \eqref{eq:nonlinearrec} satisfies the $\delta$-$\RIP_{2r}$ property. If $X$ is a local minimizer of \eqref{eq:nonlinearrec} and $Z$ is a global minimizer of \eqref{eq:nonlinearrec} with $M^*=ZZ^T$, then there exists a symmetric matrix $\mathbf H \in \RR^{n^2 \times n^2}$ such that the following two conditions hold:
\begin{enumerate}
\item $\mathbf X^T\mathbf H\mathbf e=0$;
\item $\mathbf H$ satisfies the $\delta$-$\RIP_{2r}$ property, i.e, for every matrix $U \in \RR^{n \times n}$ with $\rank(U) \leq 2r$, it holds that
\[
(1-\delta)\norm{\mathbf U}^2 \leq \mathbf U^T\mathbf H\mathbf U \leq (1+\delta)\norm{\mathbf U}^2,
\]
where $\mathbf U=\vect U$.
\end{enumerate}
\end{lemma}
\begin{proof}
Following the proof of Lemma~\ref{lem:necessary} but using a different $\mathbf H$ that will be given below, we arrive at a matrix $\xi \in \RR^{n \times n}$ satisfying $\rank(\xi) \leq 2r$ and \eqref{eq:xizero}. Choose $\mathbf H$ to be the matrix satisfying
\[
(\vect K)^T\mathbf H\vect L=[\nabla^2f(\xi)](K,L),
\]
for all $K,L \in \RR^{n \times n}$. Then, \eqref{eq:xizero} implies that $\mathbf e^T\mathbf H\mathbf X\mathbf U=0$, which further implies Condition 1 since $\mathbf U$ is arbitrary. Condition 2 immediately follows from the $\delta$-$\RIP_{2r}$ property of $f$.
\end{proof}

For given $X,Z \in \RR^{n \times r}$, one can similarly construct an optimization problem based on the conditions in Lemma~\ref{lem:necessaryfoc} as follows:
\begin{align*}
\min_{\delta,\mathbf H} \quad & \delta \\
\st \quad & \mathbf X^T\mathbf H\mathbf e=0, \\
& \text{$\mathbf H$ is symmetric and satisfies $\delta$-$\RIP_{2r}$},
\end{align*}
whose optimal value is the same as that of the following semidefinite program by an argument similar to Lemma~\ref{lem:relax}:
\begin{equation}\label{eq:ripsdpfoc}
\begin{aligned}
\min_{\delta,\mathbf H} \quad & \delta \\
\st \quad & \mathbf X^T\mathbf H\mathbf e=0, \\
& (1-\delta)I_{n^2} \preceq \mathbf H \preceq (1+\delta)I_{n^2}.
\end{aligned}
\end{equation}
Define $\delta_f(X,Z)$ to be the optimal value of the above problem. Then, the problem \eqref{eq:nonlinearrec} has no spurious local minima if the function $f$ satisfies $\delta$-$\RIP_{2r}$ such that $\delta<\delta_f(X,Z)$ for all $X,Z \in \RR^{n \times r}$ with $XX^T \neq ZZ^T$. Unfortunately, this argument cannot lead to a global guarantee since $\delta_f(0,Z)=0$ for every $Z \in \RR^{n \times r}$ corresponding to $\mathbf H=I_{n^2}$. Instead, we will turn to local guarantees on the region of all $X \in \RR^{n \times r}$ satisfying
\begin{equation}\label{eq:localcond}
\norm{XX^T-M^*}_F \leq \epsilon\lambda_r(M^*)
\end{equation}
and prove Theorem~\ref{thm:local} by further lower bounding $\delta_f(X,Z)$.

One important difference between the rank-1 and higher-rank cases is that in the latter there are infinitely many matrices $X \in \RR^{n \times r}$ that produce the same value for the matrix $XX^T$. In the proof of Theorem~\ref{thm:local}, the matrix $X$ is normalized by replacing it with another matrix $\tilde X$ with $\tilde X^TZ \succeq 0$ while keeping $XX^T=\tilde X\tilde X^T$. The reason for this normalization operation will be explained in the following lemma, which essentially says that after normalization $X$ and $Z$ cannot be too far away from each other if $XX^T$ and $ZZ^T$ are close.

\begin{lemma}[\cite{BNS2016}]\label{lem:errorlowerbound}
Let $X,Z \in \RR^{n \times r}$ be two arbitrary matrices such that $Z^TX=X^TZ$ is a positive semidefinite matrix. Then,
\[
\lambda_r(ZZ^T)\norm{Z-X}_F^2 \leq \frac{1}{2(\sqrt 2-1)}\norm{ZZ^T-XX^T}_F^2.
\]
\end{lemma}

\par\noindent{\bf Proof of Theorem~\ref{thm:local}\ }
Let $Z \in \RR^{n \times r}$ be a global minimizer of \eqref{eq:nonlinearrec} with $ZZ^T=M^*$. To prove by contradiction, assume that there exists a spurious local minimizer $X \in \RR^{n \times r}$ satisfying $XX^T \neq ZZ^T$ and \eqref{eq:localcond}. Let $X^TZ=UDV^T$ be the singular value decomposition of $X^TZ$, and define the orthogonal matrix $R=UV^T$. Therefore, the matrix
\[
(XR)^TZ=VU^TUDV^T=VDV^T \succeq 0.
\]
Furthermore, it is straightforward to verify that $XR$ is also a spurious local minimizer satisfying \eqref{eq:localcond}, so $\delta \geq \delta_f(XR,Z)$. On the other hand, applying Lemma~\ref{lem:deltaminfoc} (to be stated below) on the local minimizer $XR$ and the global minimizer $Z$ gives rise to the inequality
\[
\delta_f(XR,Z) \geq \sqrt{1-\frac{3+2\sqrt 2}{4}\epsilon^2},
\]
which is a contradiction.
\hfill\BlackBox\\[2mm]

\begin{lemma}\label{lem:deltaminfoc}
Let $X,Z \in \RR^{n \times r}$ such that $M^*=ZZ^T$, $XX^T \neq ZZ^T$, $X^TZ$ is a positive semidefinite matrix, and \eqref{eq:localcond} is satisfied for some $\epsilon \in (0,2(\sqrt 2-1)]$. Then,
\[
\delta_f(X,Z) \geq \sqrt{1-\frac{3+2\sqrt 2}{4}\epsilon^2}.
\]
\end{lemma}
\begin{proof}
The statement is obviously true when $\lambda_r(M^*)=0$. If $\lambda_r(M^*)>0$, by the Wielandt-Hoffman theorem (see \cite{Wilkinson1970}), one can write
\[
\abs{\lambda_r(XX^T)-\lambda_r(M^*)} \leq \norm{XX^T-M^*}_F \leq \epsilon\lambda_r(M^*),
\]
which implies that
\begin{equation}\label{eq:eigbound1}
\lambda_r(XX^T) \geq (1-\epsilon)\lambda_r(M^*)>0.
\end{equation}

Decompose $Z$ as $Z=c_1X+c_2W$ for some scalars $c_1$ and $c_2$, where $W \in \RR^{n \times r}$ is a matrix satisfying $\norm{W}_F=1$ and $\langle X,W \rangle=0$. Then,
\[
XX^T-ZZ^T=(1-c_1^2)XX^T-c_1c_2(XW^T+WX^T)-c_2^2WW^T.
\]
We then choose
\begin{equation}\label{eq:ydef}
Y=\frac{1-c_1^2}{2}X-c_1c_2W, \quad y=\vect Y.
\end{equation}
Since $Y$ can be written as a linear combination of $X$ and $Z$ while $X^TZ$ is symmetric by assumption, $X^TY$ is also symmetric and hence $\tr(X^TY)^2 \geq 0$. Now,
\begin{align*}
\norm{\mathbf Xy}^2&=\norm{XY^T+YX^T}_F^2 \\
&=2\tr(X^TXY^TY)+\tr(X^TY)^2+\tr(Y^TX)^2 \\
&\geq 2\tr(X^TXY^TY) \\
&\geq 2\lambda_r(X^TX)\tr(Y^TY) \\
&=2\lambda_r(XX^T)\norm{y}^2.
\end{align*}
Note that \eqref{eq:ydef} and $XX^T \neq ZZ^T$ imply that $y \neq 0$, which together with \eqref{eq:eigbound1} and the above inequality further concludes that $\mathbf Xy \neq 0$. Moreover,
\begin{align*}
\norm{\mathbf e-\mathbf Xy}&=\norm{XX^T-ZZ^T-XY^T-YX^T}_F=c_2^2\norm{WW^T}_F \leq c_2^2, \\
\norm{X-Z}_F^2&=(1-c_1)^2\norm{X}_F^2+c_2^2 \geq c_2^2.
\end{align*}
Let $\theta$ be the angle between $\mathbf e$ and $\mathbf Xy$. It follows from Lemma~\ref{lem:errorlowerbound} that
\[
\sin\theta \leq \frac{\norm{\mathbf e-\mathbf Xy}}{\norm{\mathbf e}} \leq \frac{\norm{X-Z}_F^2}{\norm{XX^T-M^*}_F} \leq \frac{\norm{XX^T-M^*}_F}{2(\sqrt 2-1)\lambda_r(M^*)} \leq \frac{\epsilon}{2(\sqrt 2-1)} \leq 1.
\]
Therefore, $\theta<\pi/2$ and
\begin{equation}\label{eq:focineq2}
\cos\theta \geq \sqrt{1-\frac{3+2\sqrt 2}{4}\epsilon^2}.
\end{equation}

Define $\eta_f(X,Z)$ as the optimal value of the optimization problem
\begin{equation}\label{eq:etaoptfoc}
\begin{aligned}
\max_{\eta,\mathbf H} \quad & \eta \\
\st \quad & \mathbf X^T\mathbf H\mathbf e=0, \\
& \eta I_{n^2} \preceq \mathbf H \preceq I_{n^2}.
\end{aligned}
\end{equation}
Similar to the proof of Lemma~\ref{lem:deltamin}, it holds that
\begin{equation}\label{eq:etaineqfoc}
\eta_f(X,Z) \geq \frac{1-\delta_f(X,Z)}{1+\delta_f(X,Z)},
\end{equation}
and it is sufficient to upper bound $\eta_f(X,Z)$ through finding a feasible solution to the dual problem of \eqref{eq:etaoptfoc} given by
\begin{equation}\label{eq:etafocdual}
\begin{aligned}
\min_{U_1,U_2,y} \quad & \tr(U_2), \\
\st \quad & \tr(U_1)=1, \\
& (\mathbf Xy)\mathbf e^T+\mathbf e(\mathbf Xy)^T=U_1-U_2, \\
& U_1 \succeq 0, \quad U_2 \succeq 0.
\end{aligned}
\end{equation}

Let
\[
M=(\mathbf Xy)\mathbf e^T+\mathbf e(\mathbf Xy)^T,
\]
in which $y$ is defined by \eqref{eq:ydef}. The matrix $M$ can be decomposed as
\[
M=[M]_+-[M]_-,
\]
where $[M]_+ \succeq 0$ and $[M]_- \succeq 0$. By Lemma~\ref{lem:eigangle},
\begin{align*}
\tr([M]_+)&=\norm{\mathbf e}\norm{\mathbf Xy}(1+\cos\theta), \\
\tr([M]_-)&=\norm{\mathbf e}\norm{\mathbf Xy}(1-\cos\theta).
\end{align*}
Therefore,
\begin{gather*}
U_1^*=\frac{[M]_+}{\tr([M]_+)}, \quad U_2^*=\frac{[M]_-}{\tr([M]_+)}, \quad y^*=\frac{y}{\tr([M]_+)}
\end{gather*}
form a feasible solution to the dual problem \eqref{eq:etafocdual}, which implies that
\[
\eta_f(X,Z) \leq \frac{\tr([M]_-)}{\tr([M]_+)}=\frac{1-\cos\theta}{1+\cos\theta}.
\]
This gives the desired result after combining it with \eqref{eq:focineq2} and \eqref{eq:etaineqfoc}.
\end{proof}

\section{Application: 1-bit Matrix Completion}\label{sec:app}

To demonstrate the effectiveness of the developed conditions for the absence of spurious solutions, in this section we will study the 1-bit matrix completion problem. This is a low-rank matrix recovery problem with nonlinear measurements that naturally arises in applications such as recommendation systems in which each user provides binary (like/dislike) observations (see \cite{DPBW2014,GPY2019}). In this problem, there is an unknown ground truth matrix $M^* \in \RR^{n \times n}$ with $M^* \succeq 0$ and $\rank(M^*)=r$. One is allowed to take independent measurements on each entry $M^*_{ij}$, where each measurement value is a binary random variable whose distribution is given by
\[
Y_{ij}=\begin{dcases*}
1 & with probability $\sigma(M^*_{ij})$, \\
0 & with probability $1-\sigma(M^*_{ij})$.
\end{dcases*}
\]
Here, $\sigma(x)$ is commonly chosen to be the sigmoid function $\ee^x/(\ee^x+1)$. Note that $M^*_{ij}$ is an arbitrary real number while the measurements $Y_{ij}$ are restricted to the binary choices 0 and 1. After a large number of measurements are taken, let $y_{ij}$ be the percentage of the measurements on the $(i,j)$th entry that are equal to 1, for every $i,j \in \{1,\dots,n\}$. To recover the ground truth matrix, consider its maximum likelihood estimator. The log-likelihood function is given by
\[
\sum_{i=1}^n\sum_{j=1}^n(y_{ij}\log(\sigma(M_{ij}))+(1-y_{ij})\log(1-\sigma(M_{ij})))=\sum_{i=1}^n\sum_{j=1}^n(y_{ij}M_{ij}-\log(1+\ee^{M_{ij}})).
\]
Therefore, the 1-bit matrix completion problem can be formulated as an optimization problem in the form \eqref{eq:nonlinearrec} with
\[
f(M)=-\sum_{i=1}^n\sum_{j=1}^n(y_{ij}M_{ij}-\log(1+\ee^{M_{ij}})).
\]
Our goal is to use Theorem~\ref{thm:local} to provide a local guarantee for the absence of spurious local minima for the above problem by showing that there is no spurious local minimizer $X \in \RR^{n \times r}$ such that $XX^T \in \bar{\mathcal B}(M^*,R)$, where
\[
\bar{\mathcal B}(M^*,R)=\{M \in \RR^{n \times n} | \norm{M-M^*}_F \leq R\}.
\]

\begin{figure}[!t]
\centering
\subfloat[$n=20$, $r=2$]{\includegraphics[width=6cm]{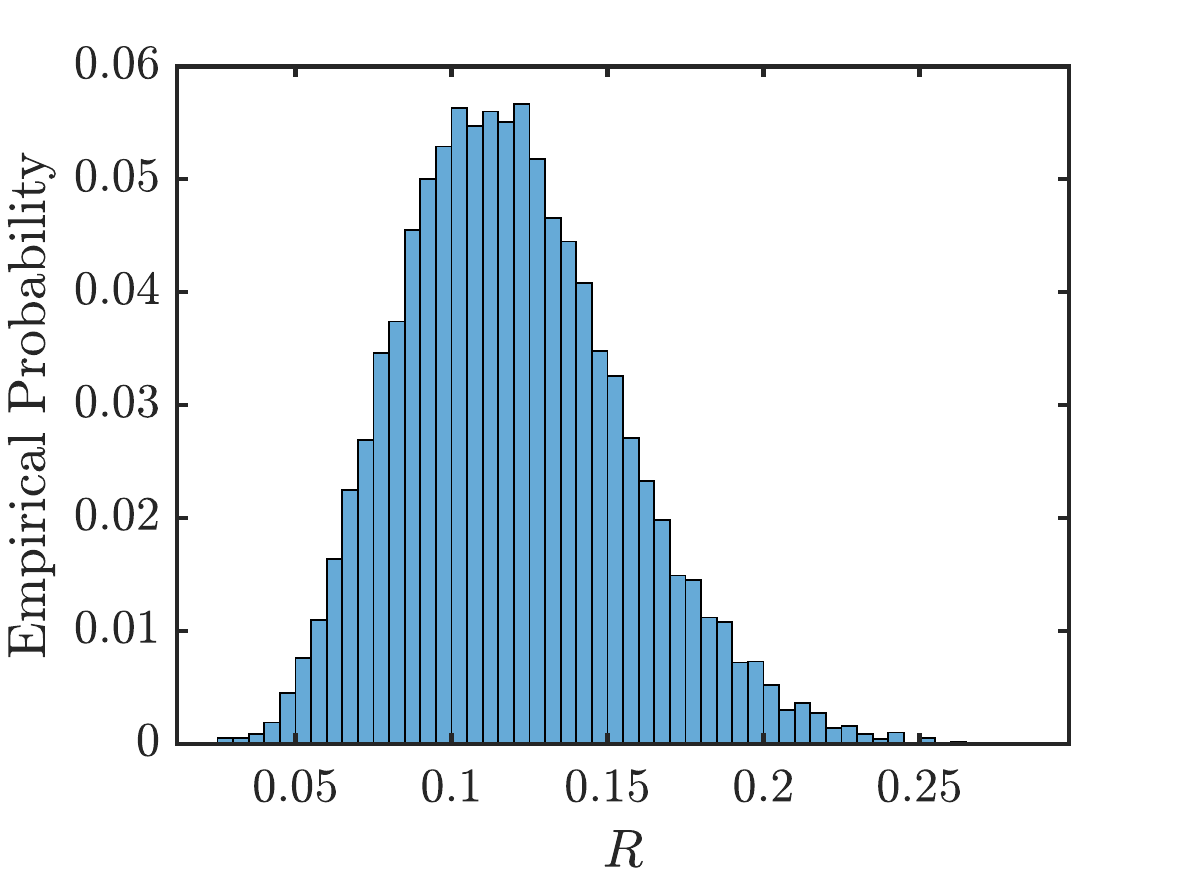}}\quad
\subfloat[$n=20$, $r=10$]{\includegraphics[width=6cm]{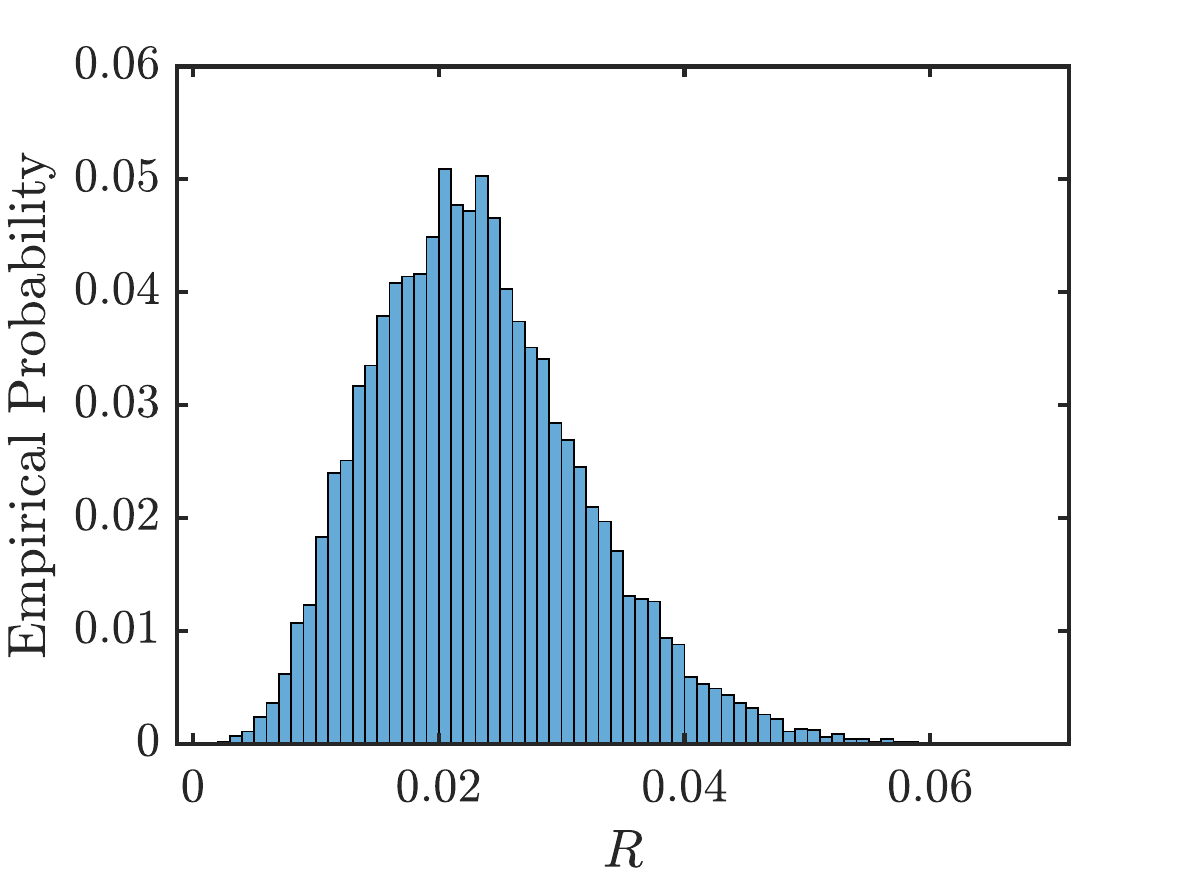}}\\
\subfloat[$n=100$, $r=2$]{\includegraphics[width=6cm]{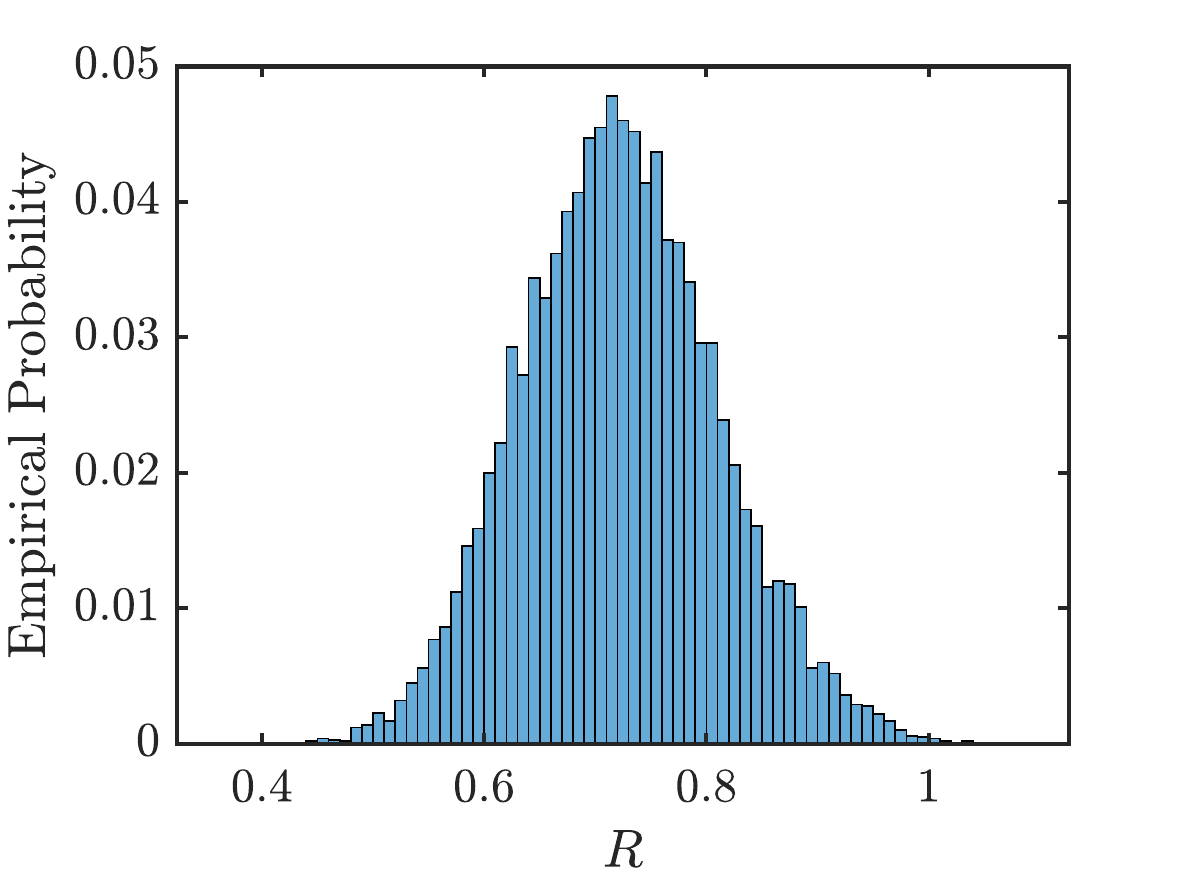}}\quad
\subfloat[$n=100$, $r=10$]{\includegraphics[width=6cm]{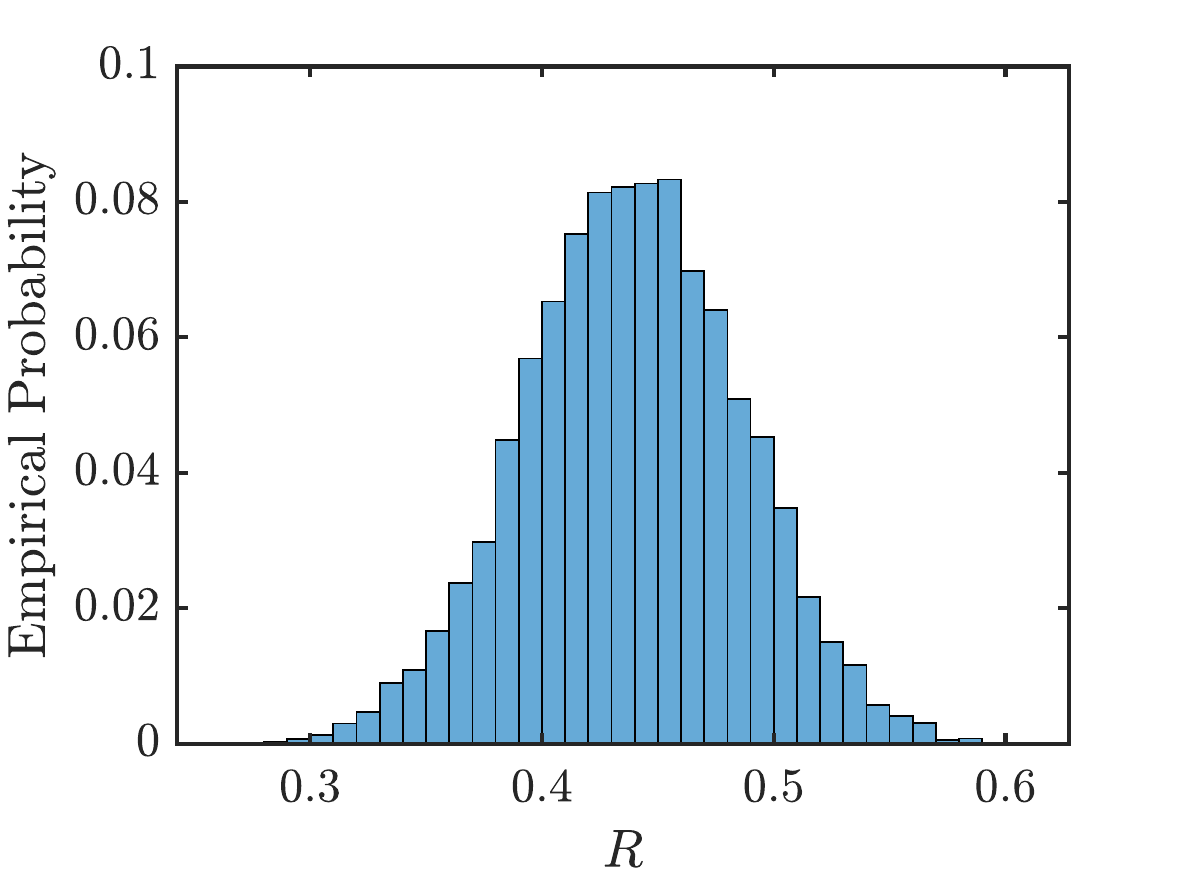}}
\caption{The empirical distribution of the radius $R$ of the neighborhood $\bar{\mathcal B}(M^*,R)$ that is guaranteed to be free of spurious local minima due to Theorem~\ref{thm:local}, for randomly generated ground truth matrices $M^*$ with different sizes $n$ and ranks $r$.}\label{fig:radius}
\end{figure}

Direct computation shows that
\begin{equation}\label{eq:onebit}
[\nabla^2f(M)](K,L)=\sum_{i=1}^n\sum_{j=1}^n\sigma'(M_{ij})K_{ij}L_{ij}, \quad \forall M,K,L \in \RR^{n \times n}.
\end{equation}
For each entry $M^*_{ij}$, define
\[
\underline M^*_{ij}=\max\{\abs{M_{ij}^*}-R,0\}, \quad \overline M^*_{ij}=\abs{M_{ij}^*}+R.
\]
Since $\sigma'(x)$ is an even function that is decreasing on the region $x \geq 0$, if we let
\[
m_1=\sigma'\left(\min_{i,j}\underline M^*_{ij}\right), \quad m_2=\sigma'\left(\max_{i,j}\overline M^*_{ij}\right),
\]
then \eqref{eq:onebit} implies that
\[
m_2\norm{K}_F^2 \leq [\nabla^2f(M)](K,K) \leq m_1\norm{K}_F^2, \quad \forall M \in \bar{\mathcal B}(M^*,R), \: K \in \RR^{n \times n}.
\]
The above inequality shows that $\gamma f(\cdot)$ satisfies $\delta$-$\RIP_{2r}$ on the region $\bar{\mathcal B}(M^*,R)$ with
\[
\gamma=\frac{2}{m_1+m_2}, \quad \delta=\frac{m_1-m_2}{m_1+m_2}.
\]
Therefore, by Theorem~\ref{thm:local}, there is no spurious local minimizer $X$ satisfying $XX^T \in \bar{\mathcal B}(M^*,R)$ as long as $R$ is sufficiently small to satisfy the inequality
\begin{equation}\label{eq:onebitbound}
\frac{m_1-m_2}{m_1+m_2}<\sqrt{1-\frac{3+2\sqrt 2}{4}\left(\frac{R}{\lambda_r(M^*)}\right)^2}.
\end{equation}
In the $r=1$ case, if we let
\[
m_3=\max_{i,j}(\sigma'(\underline M^*_{ij})-\sigma'(\overline M^*_{ij})),
\]
then \eqref{eq:onebit} also implies that
\begin{align*}
\abs{[\nabla^2f(M)-\nabla^2f(M')](K,L)} &\leq m_3\sum_{i=1}^n\sum_{j=1}^n\abs{K_{ij}L_{ij}} \\
&\leq m_3\norm{K}_F\norm{L}_F, \quad \forall M,M' \in \bar{\mathcal B}(M^*,R), \: K,L \in \RR^{n \times n},
\end{align*}
so $\gamma f(\cdot)$ satisfies $\kappa$-$\BDP_2$ on the region $\bar{\mathcal B}(M^*,R)$ with $\kappa=\gamma m_3$, and then Theorem~\ref{thm:main} can be applied similarly to obtain a possibly stronger result.

To illustrate the superiority of our result over previous ones, consider a simple special case in which $r=2$ and
\[
M^*=\diag(2,2,0,\dots,0).
\]
In this case, the inequality \eqref{eq:onebit} becomes
\[
\sigma'(2)\norm{K}_F^2 \leq [\nabla^2f(M^*)] \leq \sigma'(0)\norm{K}_F^2,
\]
which implies that $\gamma\nabla^2f(M^*)$ satisfies the $\delta$-$\RIP_4$ property for
\[
\gamma=\frac{2}{\sigma'(0)+\sigma'(2)}, \quad \delta=\frac{\sigma'(0)-\sigma'(2)}{\sigma'(0)+\sigma'(2)} \approx 0.41>1/5.
\]
In addition, it can be observed that the above choice of $\gamma$ is the best to minimize $\delta$. As a result, the existing bound given in \cite{LZT2019} cannot certify the absence of spurious local minima in the region $\bar{\mathcal B}(M^*,R)$ no matter how small $R$ is. The reason is that the function $f(\cdot)$ (after scaling) cannot satisfy the $\delta$-$\RIP_4$ property in any local neighborhood of $M^*$. On the other hand, the above discussion based on our Theorem~\ref{thm:local} shows that the problem has no spurious local minima in $\bar{\mathcal B}(M^*,R)$ as long as $R$ satisfies \eqref{eq:onebitbound}. Solving the inequality \eqref{eq:onebitbound} gives $R<1.14$.

For an arbitrary ground truth matrix $M^*$, one can perform a binary search to find the largest $R$ such that the inequality \eqref{eq:onebitbound} is satisfied and thus conclude that the problem has no spurious local minima in the neighborhood $\bar{\mathcal B}(M^*,R)$. For different sizes $n$ and ranks $r$, Figure~\ref{fig:radius} plots the empirical distribution of the radius $R$ of such neighborhood for $10^4$ random samples of $M^*=ZZ^T$ in which each entry of $Z$ is independently generated from the normal distribution with mean 0 and standard deviation 0.1.

\section{Conclusion}

In this paper, we first propose the bounded difference property (BDP) in order to study the symmetric low-rank matrix recovery problem with nonlinear measurements. The relationship between BDP and RIP is thoroughly investigated. Then, two novel criteria for the local and global nonexistence of spurious local minima are proposed. It is shown that the developed criteria are superior to the existing conditions relying solely on RIP. In particular, this work offers the first result in the literature on the nonexistence of spurious solutions in a local region for the low-rank matrix recovery problems with nonlinear measurements.

\acks{This work was supported by grants from ARO, ONR, AFOSR and NSF.}

\appendix

\section{Proof of Lemma~\ref{lem:relax}}\label{sec:app2}

Let $\OPT(X,Z)$ denote the optimal value of the optimization problem
\begin{equation}\label{eq:opt}
\begin{aligned}
\min_{\delta,\mathbf H} \quad & \delta \\
\st \quad & \norm{\mathbf X^T\mathbf H\mathbf e} \leq a, \\
& 2I_r \otimes \mat(\mathbf H\mathbf e)+\mathbf X^T\mathbf H\mathbf X \succeq -bI_{nr}, \\
& \text{$\mathbf H$ is symmetric and satisfies $\delta$-$\RIP_{2r}$},
\end{aligned}
\end{equation}
and $\LMI(X,Z)$ denote the optimal value of the optimization problem
\begin{equation}\label{eq:lmi}
\begin{aligned}
\min_{\delta,\mathbf H} \quad & \delta \\
\st \quad & \begin{bmatrix}
I_{nr} & \mathbf X^T\mathbf H\mathbf e \\
(\mathbf X^T\mathbf H\mathbf e)^T & a^2
\end{bmatrix} \succeq 0, \\
& 2I_r \otimes \mat(\mathbf H\mathbf e)+\mathbf X^T\mathbf H\mathbf X \succeq -bI_{nr}, \\
& (1-\delta)I_{n^2} \preceq \mathbf H \preceq (1+\delta)I_{n^2}.
\end{aligned}
\end{equation}
As mentioned in Section~\ref{sec:nospurious}, the first constraint in \eqref{eq:opt} and the first constraint in \eqref{eq:lmi} are interchangeable. Our goal is to prove that $\OPT(X,Z)=\LMI(X,Z)$ for given $X,Z \in \RR^{n \times r}$. Let $(v_1,\dots,v_n)$ be an orthogonal basis of $\RR^n$ such that $(v_1,\dots,v_d)$ spans the column spaces of both $X$ and $Z$. Note that $d \leq 2r$. Let $P \in \RR^{n \times d}$ be the matrix with the columns $(v_1,\dots,v_d)$ and $P_\perp \in \RR^{n \times (n-d)}$ be the matrix with the columns $(v_{d+1},\dots,v_n)$. Then,
\begin{gather*}
P^TP=I_d, \quad P_\perp^TP_\perp=I_{n-d}, \quad P_\perp^TP=0, \quad P^TP_\perp=0, \\
PP^T+P_\perp P_\perp^T=I_n, \quad PP^TX=X, \quad PP^TZ=Z.
\end{gather*}
Define $\mathbf P=P \otimes P$. Consider the auxiliary optimization problem
\begin{equation}\label{eq:aux}
\begin{aligned}
\min_{\delta,\mathbf H} \quad & \delta \\
\st \quad & \begin{bmatrix}
I_{nr} & \mathbf X^T\mathbf H\mathbf e \\
(\mathbf X^T\mathbf H\mathbf e)^T & a^2
\end{bmatrix} \succeq 0, \\
& 2I_r \otimes \mat(\mathbf H\mathbf e)+\mathbf X^T\mathbf H\mathbf X \succeq -bI_{nr}, \\
& (1-\delta)I_{d^2} \preceq \mathbf P^T\mathbf H\mathbf P \preceq (1+\delta)I_{d^2},
\end{aligned}
\end{equation}
and denote its optimal value as the function $\overline\LMI(X,Z)$. Given an arbitrary symmetric matrix $\mathbf H \in \RR^{n ^2 \times n^2}$, if $\mathbf H$ satisfies the last constraint in \eqref{eq:lmi}, then it obviously satisfies $\delta$-$\RIP_{2r}$ and subsequently the last constraint in \eqref{eq:opt}. On the other hand, if $H$ satisfies the last constraint in \eqref{eq:opt}, for every matrix $Y \in \RR^{d \times d}$ with $\mathbf Y=\vect Y$, since $\rank(PYP^T) \leq d \leq 2r$ and $\vect(PYP^T)=\mathbf P\mathbf Y$, by $\delta$-$\RIP_{2r}$ property, one arrives at
\[
(1-\delta)\norm{\mathbf Y}^2=(1-\delta)\norm{\mathbf P\mathbf Y}^2 \leq (\mathbf P\mathbf Y)^T\mathbf H\mathbf P\mathbf Y \leq (1+\delta)\norm{\mathbf P\mathbf Y}^2=(1+\delta)\norm{\mathbf Y}^2,
\]
which implies that $\mathbf H$ satisfies the last constraint in \eqref{eq:aux}. The above discussion implies that
\[
\LMI(X,Z) \geq \OPT(X,Z) \geq \overline\LMI(X,Z).
\]
Let
\[
\hat X=P^TX, \quad \hat Z=P^TZ.
\]
Lemma~\ref{lem:ineq1} and Lemma~\ref{lem:ineq2} to be stated later will show that
\[
\LMI(X,Z) \leq \LMI(\hat X,\hat Z) \leq \overline\LMI(X,Z),
\]
which completes the proof of Lemma~\ref{lem:relax}.

Before stating Lemma~\ref{lem:ineq1} and Lemma~\ref{lem:ineq2} that were needed in the proof of Lemma~\ref{lem:relax}, we should first state a preliminary result below.

\begin{lemma}
Define $\hat{\mathbf e}$ and $\hat{\mathbf X}$ in the same way as $\mathbf e$ and $\mathbf X$, except that $X$ and $Z$ are replaced by $\hat X$ and $\hat Z$, respectively. Then, it holds that
\begin{align*}
\mathbf e&=\mathbf P\hat{\mathbf e}, \\
\mathbf X(I_r \otimes P)&=\mathbf P\hat{\mathbf X}, \\
\mathbf P^T\mathbf X&=\hat{\mathbf X}(I_r \otimes P)^T.
\end{align*}
\end{lemma}
\begin{proof}
Observe that
\begin{gather*}
\mathbf e=\vect(XX^T-ZZ^T)=\vect(P(\hat X\hat X^T-\hat Z\hat Z^T)P^T)=\mathbf P\hat{\mathbf e}, \\
\begin{aligned}
\mathbf X(I_r \otimes P)\vect\hat U&=\mathbf X\vect(P\hat U)=\vect(X\hat U^TP^T+P\hat UX^T) \\
&=\vect(P(\hat X\hat U^T+\hat U\hat X^T)P^T)=\mathbf P\hat{\mathbf X}\vect\hat U,
\end{aligned}\\
\begin{aligned}
\hat{\mathbf X}(I_r \otimes P)^T\vect U&=\hat{\mathbf X}\vect(P^TU)=\vect(\hat XU^TP+P^TU\hat X^T) \\
&=\vect(P^T(XU^T+UX^T)P)=\mathbf P^T\mathbf X\vect U,
\end{aligned}
\end{gather*}
where $U \in \RR^{n \times r}$ and $\hat U \in \RR^{d \times r}$ are arbitrary matrices.
\end{proof}

\begin{lemma}\label{lem:ineq1}
The inequality $\LMI(\hat X,\hat Z) \geq \LMI(X,Z)$ holds.
\end{lemma}
\begin{proof}
Let $(\delta,\hat{\mathbf H})$ be an arbitrary feasible solution to the optimization problem defining $\LMI(\hat X,\hat Z)$ with $\delta \leq 1$. It is desirable to show that $(\delta,\mathbf H)$ with
\[
\mathbf H=\mathbf P\hat{\mathbf H}\mathbf P^T+(I_{n^2}-\mathbf P\mathbf P^T)
\]
is a feasible solution to the optimization problem defining $\LMI(X,Z)$, which directly proves the lemma. To this end, notice that
\[
\mathbf H-(1-\delta)I_{n^2}=\mathbf P(\hat{\mathbf H}-(1-\delta)I_{d^2})\mathbf P^T+\delta(I_{n^2}-\mathbf P\mathbf P^T),
\]
which is positive semidefinite because
\begin{align*}
I_{n^2}-\mathbf P\mathbf P^T&=(PP^T+P_\perp P_\perp^T) \otimes (PP^T+P_\perp P_\perp^T)-(PP^T) \otimes (PP^T) \\
&=(PP^T) \otimes (P_\perp P_\perp^T)+(P_\perp P_\perp^T) \otimes (PP^T)+(P_\perp P_\perp^T) \otimes (P_\perp P_\perp^T) \succeq 0.
\end{align*}
Similarly,
\[
\mathbf H-(1+\delta)I_{n^2} \preceq 0,
\]
and therefore the last constraint in \eqref{eq:lmi} is satisfied and $\mathbf H$ is always positive semidefinite. Next, since
\[
\mathbf X^T\mathbf H\mathbf e=\mathbf X^T\mathbf H\mathbf P\hat{\mathbf e}=\mathbf X^T\mathbf P\hat{\mathbf H}\hat{\mathbf e}=(I_r \otimes P)\hat {\mathbf X}^T\hat{\mathbf H}\hat{\mathbf e},
\]
we have
\[
\norm{\mathbf X^T\mathbf H\mathbf e}^2=(\hat {\mathbf X}^T\hat{\mathbf H}\hat{\mathbf e})^T(I_r \otimes P^T)(I_r \otimes P)(\hat {\mathbf X}^T\hat{\mathbf H}\hat{\mathbf e})=\norm{\hat {\mathbf X}^T\hat{\mathbf H}\hat{\mathbf e}}^2,
\]
and thus the first constraint in \eqref{eq:lmi} is satisfied. Finally, by letting $W \in \RR^{d \times d}$ be the vector satisfying $\vect W=\hat{\mathbf H}\hat{\mathbf e}$, one can write
\[
\vect(PWP^T)=\mathbf P\vect W=\mathbf P\hat{\mathbf H}\hat{\mathbf e}.
\]
Hence,
\begin{align*}
2I_r \otimes \mat(\mathbf H\mathbf e)&=2I_r \otimes \mat(\mathbf H\mathbf P\hat{\mathbf e})=2I_r \otimes \mat(\mathbf P\hat{\mathbf H}\hat{\mathbf e})=I_r \otimes (P(W+W^T)P^T) \\
&=2I_r \otimes (P\mat(\hat{\mathbf H}\hat{\mathbf e})P^T)=2(I_r \otimes P)(I_r \otimes \mat(\hat{\mathbf H}\hat{\mathbf e}))(I_r \otimes P)^T.
\end{align*}
In addition,
\[
\mathbf X^T\mathbf H\mathbf X(I_r \otimes P)=\mathbf X^T\mathbf H\mathbf P\hat{\mathbf X}=\mathbf X^T\mathbf P\hat{\mathbf H}\hat{\mathbf X}=(I_r \otimes P)\hat{\mathbf X}^T\hat{\mathbf H}\hat{\mathbf X}.
\]
Therefore, by defining
\[
\mathbf S:=2I_r \otimes \mat(\mathbf H\mathbf e)+\mathbf X^T\mathbf H\mathbf X+bI_{nr},
\]
we have
\begin{align*}
(I_r \otimes P)^T\mathbf S(I_r \otimes P)&=2I_r \otimes \mat(\hat{\mathbf H}\hat{\mathbf e})+\hat{\mathbf X}^T\hat{\mathbf H}\hat{\mathbf X}+bI_{dr} \succeq 0, \\
(I_r \otimes P_\perp)^T\mathbf S(I_r \otimes P_\perp)&=(I_r \otimes P_\perp)^T\mathbf X^T\mathbf H\mathbf X(I_r \otimes P_\perp)+bI_{(n-d)r} \succeq 0, \\
(I_r \otimes P_\perp)^T\mathbf S(I_r \otimes P)&=0.
\end{align*}
Since the columns of $I_r \otimes P$ and $I_r \otimes P_\perp$ form a basis for $\RR^{nr}$, the above inequalities imply that $\mathbf S$ is positive semidefinite, and thus the second constraint in \eqref{eq:lmi} is satisfied.
\end{proof}

\begin{lemma}\label{lem:ineq2}
The inequality $\overline\LMI(X,Z) \geq \LMI(\hat X,\hat Z)$ holds.
\end{lemma}
\begin{proof}
The dual problem of the optimization problem defining $\LMI(\hat X,\hat Z)$ can be expressed as
\begin{equation}\label{eq:lmidual}
\begin{aligned}
\max_{\hat U_1,\hat U_2,\hat V,\hat G,\hat\lambda,\hat y} \quad & \tr(\hat U_1-\hat U_2)-\tr(\hat G)-a^2\hat\lambda-b\tr(\hat V) \\
\st \quad & \tr(\hat U_1+\hat U_2)=1, \\
&\sum_{j=1}^r(\hat{\mathbf X}\hat y-\vect\hat V_{j,j})\hat{\mathbf e}^T+\sum_{j=1}^r\hat{\mathbf e}(\hat{\mathbf X}\hat y-\vect\hat V_{j,j})^T-\hat{\mathbf X}\hat V\hat{\mathbf X}^T=\hat U_1-\hat U_2, \\
& \begin{bmatrix}
\hat G & -\hat y \\
-\hat y^T & \hat \lambda
\end{bmatrix} \succeq 0, \\
&\hat U_1 \succeq 0, \quad \hat U_2 \succeq 0, \quad \hat V=\begin{bmatrix}
\hat V_{1,1} & \cdots & \hat V_{r,1} \\
\vdots & \ddots & \vdots \\
\hat V_{r,1}^T & \cdots & \hat V_{r,r}
\end{bmatrix} \succeq 0.
\end{aligned}
\end{equation}
Since
\[
\hat U_1=\frac{1}{2d^2}I_{d^2}-\frac{\mu}{2}M, \quad \hat U_2=\frac{1}{2d^2}I_{d^2}+\frac{\mu}{2}M, \quad \hat V=\mu I_{dr}, \quad \hat G=I_{dr}, \quad \hat\lambda=1, \quad \hat y=0,
\]
where
\[
M=r((\vect I_d)\hat{\mathbf e}^T+\hat{\mathbf e}(\vect I_d)^T)+\hat{\mathbf X}\hat{\mathbf X}^T,
\]
is a strict feasible solution to the above dual problem \eqref{eq:lmidual} as long as $\mu>0$ is sufficiently small, Slater's condition implies that strong duality holds for the optimization problem defining $\LMI(\hat X,\hat Z)$. Therefore, we only need to prove that the optimal value of \eqref{eq:lmidual} is smaller than or equal to the optimal value of the dual of the optimization problem defining $\overline\LMI(X,Z)$ given by:
\begin{equation}\label{eq:lmidual2}
\begin{aligned}
\max_{U_1,U_2,V,G,\lambda,y} \quad & \tr(U_1-U_2)-\tr(G)-a^2\lambda-b\tr(V) \\
\st \quad & \tr(U_1+U_2)=1, \\
&\sum_{j=1}^r(\mathbf Xy-\vect V_{j,j})\mathbf e^T+\sum_{j=1}^r\mathbf e(\mathbf Xy-\vect V_{j,j})^T-\mathbf XV\mathbf X^T=\mathbf P(U_1-U_2)\mathbf P^T, \\
& \begin{bmatrix}
G & -y \\
-y^T & \lambda
\end{bmatrix} \succeq 0, \\
&U_1 \succeq 0, \quad U_2 \succeq 0, \quad V=\begin{bmatrix}
V_{1,1} & \cdots & V_{r,1} \\
\vdots & \ddots & \vdots \\
V_{r,1}^T & \cdots & V_{r,r}
\end{bmatrix} \succeq 0.
\end{aligned}
\end{equation}
The above claim can be verified by noting that given any feasible solution
\[
(\hat U_1,\hat U_2,\hat V,\hat G,\hat\lambda,\hat y)
\]
to \eqref{eq:lmidual}, the matrices
\begin{gather*}
U_1=\hat U_1, \quad U_2=\hat U_2, \quad V=(I_r \otimes P)\hat V(I_r \otimes P)^T, \\
\begin{bmatrix}
G & -y \\
-y^T & \lambda
\end{bmatrix}=\begin{bmatrix}
I_r \otimes P & 0 \\
0 & 1
\end{bmatrix}\begin{bmatrix}
\hat G & -\hat y \\
-\hat y^T & \hat\lambda
\end{bmatrix}\begin{bmatrix}
(I_r \otimes P)^T & 0 \\
0 & 1
\end{bmatrix}
\end{gather*}
form a feasible solution to \eqref{eq:lmidual2}, and both solutions have the same optimal objective value.
\end{proof}

\vskip 0.2in
\bibliography{nonlinearrip}
\end{document}